\documentclass[a4paper,10pt]{article}

\usepackage[affil-sl]{authblk}
\usepackage[linesnumbered,ruled,lined,commentsnumbered]{algorithm2e}
\usepackage{amsmath}
\usepackage{amssymb}
\usepackage{amsthm}
\usepackage{blkarray}
\usepackage{dsfont}
\usepackage{enumerate}
\usepackage[inline]{enumitem}
\usepackage[margin=3cm]{geometry}
\usepackage{mathtools}
\usepackage{nicefrac}
\usepackage{tikz}
\usetikzlibrary{calc}
\usetikzlibrary{trees}
\usetikzlibrary{decorations.pathmorphing, decorations.pathreplacing, decorations.shapes}
\usepackage[binary-units=true]{siunitx}
\usepackage[textsize=tiny]{todonotes}
\usepackage{wrapfig}
\usepackage{booktabs}
\usepackage{subfigure}

\definecolor{darkred}{RGB}{150, 0, 0}
\definecolor{darkgreen}{RGB}{0, 150, 0}
\definecolor{darkblue}{RGB}{0, 0, 150}
\usepackage[breaklinks=true,colorlinks,citecolor=darkgreen,linkcolor=darkred,urlcolor=darkblue,bookmarks=false]{hyperref}

\newtheorem{theorem}{Theorem}
\newtheorem{lemma}[theorem]{Lemma}
\newtheorem{remark}[theorem]{Remark}

\newtheorem{proposition}[theorem]{Proposition}

\theoremstyle{definition}
\newtheorem{example}[theorem]{Example}
\newtheorem{observation}[theorem]{Observation}

\newcommand{\R}{\mathds{R}}
\newcommand{\Z}{\mathds{Z}}
\newcommand{\B}[1]{\{0,1\}^{#1}}
\newcommand{\cube}[1]{[0, 1]^{#1}}

\newcommand{\solver}[1]{\textsc{#1}}
\newcommand{\scip}{\solver{SCIP}\xspace}

\newcommand{\sparsity}{\ensuremath{\sigma}}
\newcommand{\K}{w_3}
\newcommand{\Q}{w_2}
\newcommand{\J}{w_1}

\newcommand{\myalpha}[3]{\left(\alpha^{ #1 }_{ #2 }\right)^{ #3 }}
\newcommand{\Up}{\texttt{U}}
\newcommand{\Down}{\texttt{L}}
\newcommand{\Neutral}{\texttt{N}}
\newcommand{\myw}{v}

\newcommand{\orbisack}{O}
\newcommand{\orbisacktrans}{\bar{O}}

\DeclareMathOperator{\conv}{conv}

\renewcommand{\log}[2][]{\text{log}_{#1}\left( #2 \right)}

\newcommand{\T}{^\top}
\newcommand{\sprod}[2]{{#1}\T{#2}}
\newcommand{\define}{\coloneqq}
\newcommand{\card}[1]{\left\lvert #1 \right\rvert}

\newcommand{\ceil}[1]{\left\lceil #1 \right\rceil}
\newcommand{\brackets}[1]{\left\{ #1 \right\}}
\newcommand{\parentheses}[1]{\left( #1 \right)}

\newcommand{\naturalsto}[2][]{\left[ #2 \right]_{#1}}
\newcommand{\Oh}[1]{\mathcal{O}\parentheses{ #1 }}

\DeclareMathOperator{\argmax}{argmax}
\DeclareMathOperator{\argmin}{argmin}

\newcommand{\mygreen}{darkgreen}
\newcommand{\myred}{red}
\newcommand{\myblue}{blue}

\author[1]{Christopher Hojny}
\author[1]{C\'edric Roy}
\affil[1]{%
  Eindhoven University of Technology,
  Eindhoven, The Netherlands\\

  \emph{email} \{c.hojny, c.j.roy\}@tue.nl
}

\begin{document}

\title{Computational Aspects of Lifted Cover Inequalities for Knapsacks with Few Different Weights\footnote{This article is part of the project ``Local Symmetries for Global Success'' with project number OCENW.M.21.299, which is financed by the Dutch Research Council (NWO).}}

\date{}
\maketitle

\begin{abstract}
    Cutting planes are frequently used for solving integer programs. 
    A common strategy is to derive cutting planes from building blocks or a substructure of the integer program. 
    In this paper, we focus on knapsack constraints that arise from single row relaxations.
    Among the most popular classes derived from knapsack constraints are lifted minimal cover inequalities.
    The separation problem for these inequalities is NP-hard though, and one usually separates them heuristically, therefore not fully exploiting their potential.

    For many benchmarking instances however, it turns out that many knapsack constraints only have few different coefficients. 
    This motivates the concept of sparse knapsacks where the number of different coefficients is a small constant, independent of the number of variables present.
    For such knapsacks, we observe that there are only polynomially many different classes of structurally equivalent minimal covers.
    This opens the door to specialized techniques for using lifted minimal cover inequalities.

    In this article we will discuss two such techniques, which are based on specialized sorting methods.
    On the one hand, we present new separation routines that separate equivalence classes of inequalities rather than individual inequalities.
    On the other hand, we derive compact extended formulations that express all lifted minimal cover inequalities by means of a polynomial number of constraints.
    These extended formulations are based on tailored sorting networks that express our separation algorithm by linear inequalities.
    We conclude the article by a numerical investigation of the different techniques for popular benchmarking instances.
\end{abstract}

\newpage
\section{Introduction}

\noindent
We consider binary programs~$\max \{ \sprod{d}{x} : Ax \leq b,\; x \in \B{n}\}$, where~$A \in \R^{m \times n}$, $b \in \R^m$, and~$d \in \R^n$.
A standard technique to solve such problems is branch-and-bound~\cite{doig1960algorithms}.
Among the many techniques to enhance branch-and-bound, one popular class are cutting planes.
These are inequalities~$\sprod{c}{x} \leq \delta$ that are satisfied by each feasible solution of the binary program, but which exclude some points of the LP relaxation.
Cutting planes turn out to be a crucial component of modern branch-and-bound solvers, since disabling them may degrade the performance drastically~\cite{bixby2012brief}.

Many families of cutting planes are known in the literature.
In this article, we focus on cutting planes arising from knapsack polytopes, which are among the most extensively studied \cite{balas1975facets,boyd1993generating,boyd1994fenchel,crowder1983solving,gu1999lifted,hojny2020knapsack,wolsey2014integer}.
A \emph{knapsack set} is a set~$K^{a,\beta} = \brackets{ x \in \B{n} : \sprod{a}{x} \leq \beta }$ for some non-negative vector~$a \in \Z^n$ and positive integer~$\beta$; the corresponding \emph{knapsack polytope} is~$P^{a, \beta} = \conv(K^{a, \beta})$, where~$\conv(\cdot)$ denotes the convex hull operator.
Note that any cutting plane derived from knapsack sets can be used for general binary programs by considering a single row of the inequality system~$Ax \leq b$ (after possibly complementing some variables).
A popular class of knapsack-based cutting planes are derived from so-called covers.
A \emph{cover} is a set~$C\subseteq [n] \define \{1,\dots,n\}$ with~$\sum_{i\in C} a_i > \beta$.
The corresponding \emph{cover inequality}~\cite{balas1975facets,balas1978facets,wolsey1975faces} is~$\sum_{i \in C} x_i \leq \card{C} - 1$, which implies that not all elements in~$C$ can simultaneously attain value~1.
It is easy to show that, given two covers~$C, C'$, the cover inequality for~$C$ can be dominated by the inequality for~$C'$ if~$C'\subsetneq C$.
This motivates to consider covers~$C$ that are \emph{minimal}, i.e., no proper subset of~$C$ is a cover.
To strengthen these inequalities even further, so-called sequential lifting~\cite{padberg1975note} can be used to turn a cover inequality for a minimal cover~$C$ into a facet-defining inequality
\begin{equation}
    \label{eq:liftedCover}
    \sum_{i \in C} x_i
    +
    \sum_{i \in [n] \setminus C} \alpha_i x_i
    \leq
    \card{C} - 1
\end{equation}
for the knapsack polytope, i.e, the inequality cannot be dominated by other inequalities.

To use lifted cover inequalities (LCIs) as cutting planes, one could, in principle, fully enumerate and add them to the binary program.
However, since there might be exponentially many covers, this is practically infeasible.
Alternatively, one could add violated LCIs dynamically during the solving process.
Deciding whether a violated LCI exists, is NP-complete~\cite{del2023complexity} though.
In practice, one therefore usually adds violated LCIs heuristically~\cite{kaparis2010separation}.
For many knapsacks arising from the MIPLIB~2017~\cite{MIPLIB} test set, however, we made an important observation:
they only have very few different coefficients, say less than five.
Indeed, as illustrated by Figure~\ref{fig:distributionKnapsack},
approximately~\SI{41}{\percent} of all knapsacks have at most~4 different
coefficients\footnote{In~\ref{sec:appendix}, we provide details on how
  these numbers have been found.}.
Moreover, although the number of different coefficients can be small, the
knapsacks still can contain hundreds or thousands of variables, e.g.,
\SI{40}{\percent} of knapsacks with only four different coefficients
contain more than 120 variables, see Figure~\ref{fig:lengthKnapsack}.
To the best of our knowledge, this structure is not exploited in integer
programming solvers.

\begin{figure}[t]
  \centering
  \subfigure[Absolute frequency.]{
    \includegraphics[scale=0.46]{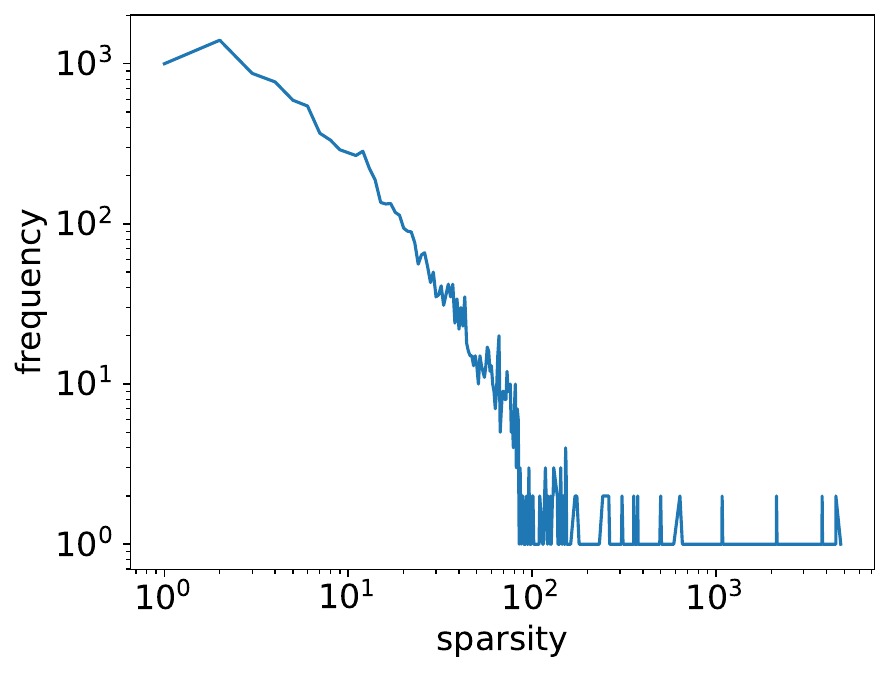}
  }
  \subfigure[Cumulative frequency in percent.]{
    \includegraphics[scale=0.46]{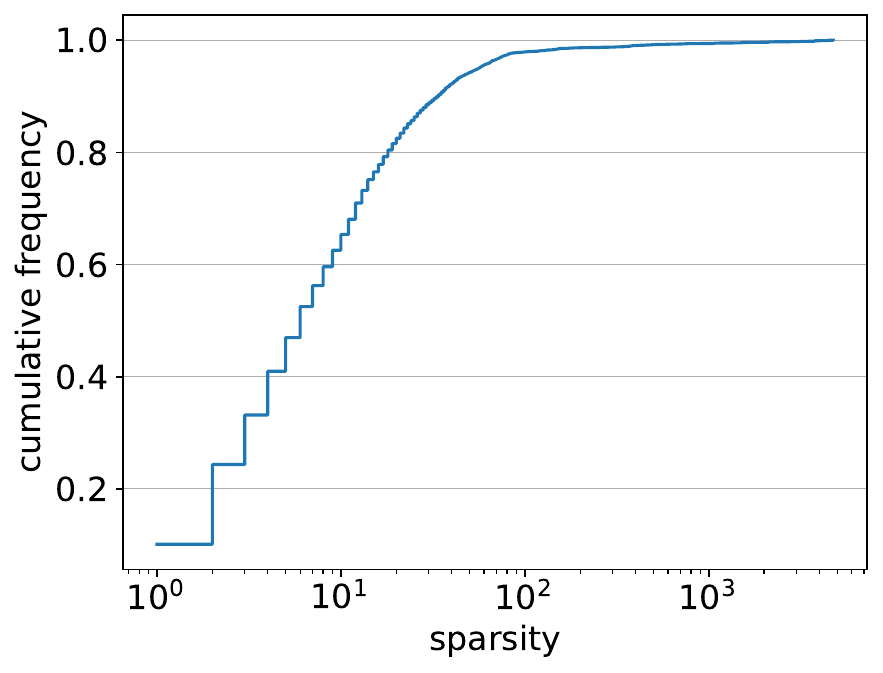}
  }
  \caption{Illustration of the frequencies at which sparse knapsacks arise
    in MIPLIB~2017.
    The horizontal axis shows the sparsity and the vertical axis plots the
    (cumulative) frequency of the respective sparsity level.
  }
  \label{fig:distributionKnapsack}
\end{figure}

\begin{figure}[t]
  \centering
  \subfigure[Sparsity 3.]{
    \includegraphics[scale=0.45]{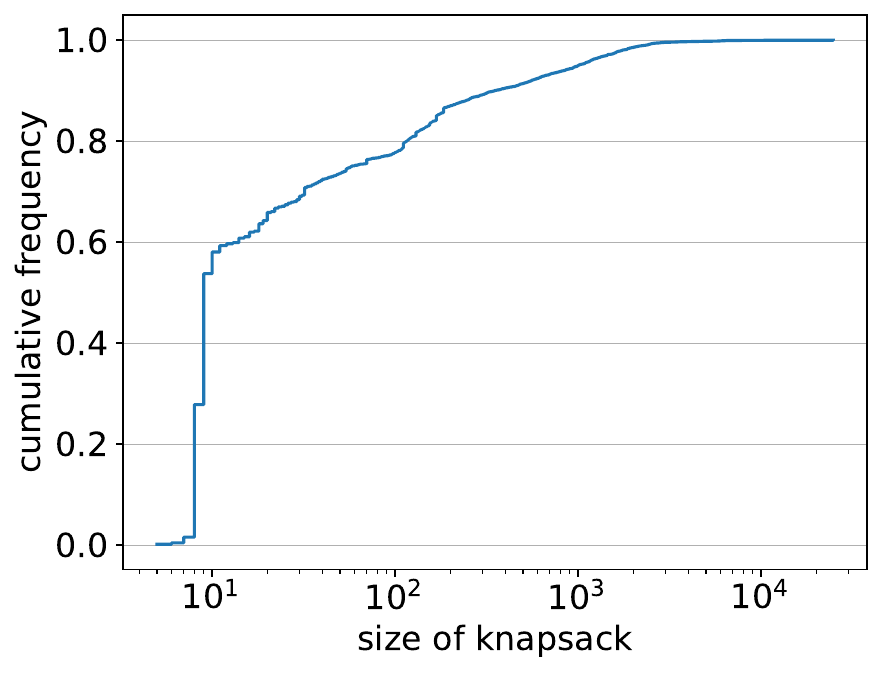}
  }
  \subfigure[Sparsity 4.]{
    \includegraphics[scale=0.45]{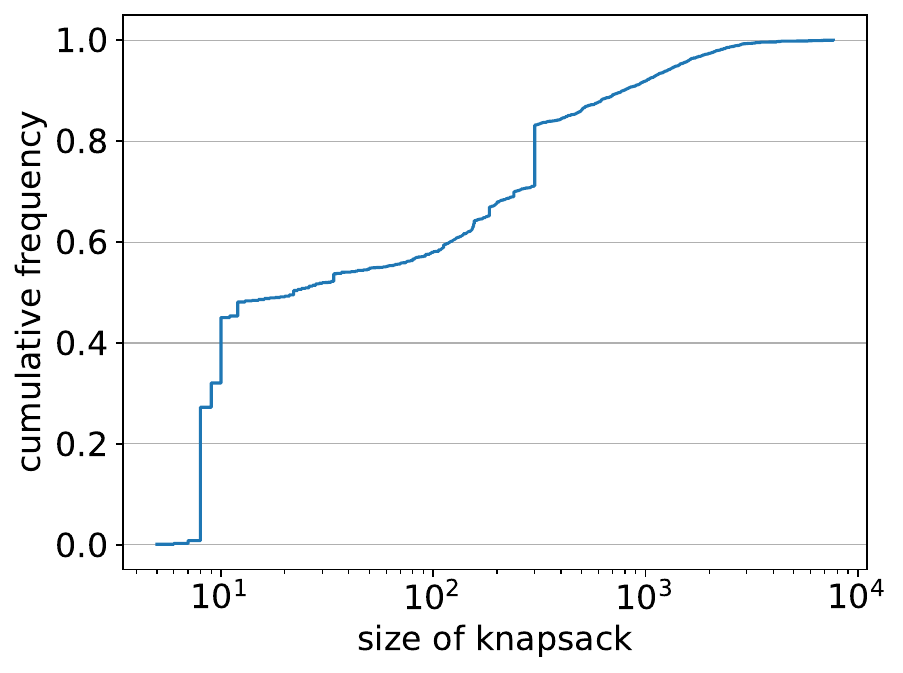}
  }
  \caption{For fixed sparsity of~3 and~4, illustration of the cumulative
    frequency at which knapsacks of a specific size arise in MIPLIB~2017.
    The horizontal axis shows the size of the knapsack (number of
    variables) and the vertical axis plots the cumulative frequency of
    knapsacks of this size.
  }
  \label{fig:lengthKnapsack}
\end{figure}

We therefore investigate so-called sparse knapsacks in this article.
A knapsack with inequality~$\sum_{i = 1}^n a_i x_i \leq \beta$ is called \emph{$\sparsity$-sparse} if the number of different coefficients is at most~\sparsity.
After introducing some notation, in Section~\ref{sec:lifted-cover-inequality-separation-for-sparse-knapsack} we show how the simplified structure of sparse knapsacks allows for solving the separation problem for LCIs in polynomial time (Theorem~\ref{thm:complexitySparse}).
In Section~\ref{sec:Sorting-network} we propose a polyhedral model for this the separation procedure, using sorting networks.
We have implemented our techniques for sparse knapsacks in the academic solver \scip~\cite{bolusani2024scip} and give an overview of it in Section~\ref{chap:practical-implementation}.
In Section~\ref{sec:numerics}, we report on numerical experience, showing, among others, that exactly separating LCIs for sparse knapsacks can substantially improve the performance of \scip.

\paragraph{Related Literature}
In the following we provide an overview of cutting planes derived from knapsack polytopes.
We refer the reader to the survey~\cite{hojny2020knapsack} for a more detailed discussion.
Deriving inequalities from covers~\cite{balas1975facets,balas1978facets,wolsey1975faces} is a well-known topic in the domain of integer programming. 
These cover inequalities can be strengthened by lifting all the coefficients for variables not in~$C$. 
There exist facet-defining lifting sequences~\cite{padberg1975note,zemel1989easily}, so-called \emph{down-lifting} sequences~\cite{chen2021complexity,wolsey2014integer}, or even simultaneous lifting procedures~\cite{easton2008simultaneously,gu1999lifted,letchford2019lifted,marchand2002cutting,prasad2024newsequenceindependentliftingtechniques,wolsey1977valid}.
Additionally, there also exist lifting techniques for variations of the original problem, such as liftings for non-minimal covers~\cite{letchford2019lifted} or liftings for 0/1-coefficient polytopes~\cite{peled1977properties}.
Balas and Zemel~\cite{balas1978facets} gave a complete description of the facet-defining inequalities arising from lifted cover inequalities.
Deciding whether a given inequality is an LCI is polynomial time~\cite{hartvigsen1992complexity}, but the problem of separating cutting planes for knapsack polytopes is known to be NP-complete~\cite{del2023complexity,fereirra1994combinatorial,klabjan1998complexity}.
For this reason, LCIs are usually separated heuristically~\cite{gu1999lifted,hoffman1991improving}.
Next to LCIs, further cutting planes are discussed, among others, merged cover inequalities~\cite{hickman2015merging}, $(1, k)$-configurations~\cite{padberg19801,gottlieb1988facets}, coefficient increased cover inequalities~\cite{dietrich1992tightening}, lifted pack inequalities~\cite{atamturk2005cover,weismantel19970}, weight inequalities~\cite{weismantel19970}, Gomory cuts~\cite{glover1997generating}, and exact separation~\cite{andrew1992pseudopolynomial,boyd1993generating,boyd1994fenchel}.

\paragraph{Basic Definitions and Notation}
Just as we use~$\naturalsto{n}$ as shorthand for the set of positive naturals~$\brackets{1, \dots, n}$, let~$\naturalsto[0]{n}\define\naturalsto{n}\cup\brackets{0}$.
Without loss of generality, all the knapsack constraints we will discuss will neither be trivial, thus implicitly satisfying $\sum_{i=1}^n a_i > \beta$, nor have trivial variables, which means $0 < a_i \leq \beta$ for all $i\in\naturalsto{n}$.
Given a set of values~$\brackets{a_1, \dots, a_n}$ and a set of indices~$C\subseteq\naturalsto{n}$, we use the shorthand~$a(C)\define\sum_{i\in C} a_i$.
Similarly, for a permutation~$\gamma$ of~$\naturalsto{n}$, we denote~$\gamma(C)\define\brackets{\gamma(i):i\in C}$.

\section{Lifted Cover Inequality Separation for Sparse Knapsacks}
\label{sec:lifted-cover-inequality-separation-for-sparse-knapsack}

\noindent
Throughout this section, let~$a \in \Z_+^n$ and~$\beta \in \Z_+$.
To make use of LCIs for the knapsack~$K^{a,\beta}$ when solving binary programs, we have mentioned two approaches in the introduction.
One either explicitly enumerates all minimal covers and computes all their liftings, or one adds LCIs dynamically during the solving process.
The latter approach requires to solve the so-called \emph{separation problem}, i.e., given a vector~$\bar{x} \in \R^n$, we need to decide whether there exists an LCI which is violated by~$\bar{x}$.
For general knapsacks, both approaches have their drawbacks:
explicit enumeration may need to find exponentially many minimal covers, and solving the separation problem is NP-complete~\cite{klabjan1998complexity} in general.

Based on our observation that many knapsacks in instances from MIPLIB~2017 are sparse, this section's goal is to understand the complexity of separating LCIs for sparse knapsacks.
The main insight of this section is that the separation problem can be
solved in polynomial time.
Although the proof is not difficult, we are not aware of any reference
explicitly discussing this case.
To be self-contained, we provide here a full proof, which also introduces the concepts needed in Section~\ref{sec:Sorting-network}.
\begin{theorem}
  \label{thm:complexitySparse}
  Let~$a \in \Z_+^n$ and let~$\beta, \sparsity$ be positive integers such
  that~$a$ is~$\sparsity$-sparse.
  Then, the separation problem of LCIs for~$K^{a, \beta}$ can be solved in~$\Oh{\sparsity^2 n^{2\sparsity}}$.
\end{theorem}
This result complements other results on polynomial cases of the separation problem, namely separating variants of LCIs for points~$\bar{x}$ with a constant number of non-integral entries~\cite{del2023complexity}.
That is, only constantly many entries of~$\bar{x}$ are non-zero.

The rest of this section is structured as follows.
We start by providing an explicit definition of LCIs in Section~\ref{sec:backgroundLCI}.
Afterward, Section~\ref{sec:proofComplexity} provides the proof of Theorem~\ref{thm:complexitySparse}.

\subsection{Background on Lifted Cover Inequalities}
\label{sec:backgroundLCI}

\noindent
Let~$C$ be a minimal cover of the knapsack~$K^{a,\beta}$.
Recall that its minimal cover inequality is given by~$x(C) \leq \card{C} - 1$.
In general, this inequality can be weak.
To possibly turn it into a stronger inequality, one can assign coefficients~$\alpha_i$ to the variables~$x_i$ not contained in the cover~$C$, leading to an inequality
\begin{equation}
    \sum_{i \in C} x_i + \sum_{i\notin C} \alpha_i x_i \leq |C| - 1.
    \label{eq:basic_LCI}
\end{equation}
The approach of finding the values of~$\alpha_i$ is called lifting.
Among many existing methods to get these coefficients \cite{balas1975facets,chen2021complexity,easton2008simultaneously,letchford2019lifted,marchand2002cutting,padberg1975note,zemel1989easily}, we will focus on the so-called \emph{sequential lifting} procedure that is guaranteed to yield LCIs that define facets of the knapsack polytope~$P^{a,\beta}$.
This procedure has been developed in~\cite{balas1975facets,padberg1973facial} to define some lifting coefficients.
Later, \cite{balas1978facets,nemhauser1994lifted} provide a full characterization for computing simultaneously all lifting coefficients that yield facet defining inequalities.

To describe the characterization of lifting coefficients, assume~$C = \{j_1,\dots,j_{\card{C}}\}$ such that~$a_{j_i} \geq a_{j_{i+1}}$ for all~$i \in [\card{C} - 1]$. 
For any non-negative integer~$h$, we let~$\mu(h)$ be the sum of the~$h$ heaviest elements in the cover, i.e.,
\begin{equation}
    \mu(h) \define \sum_{i = 1}^{\min\{h, \card{C}\}} a_{j_i}.
    \label{eq:def_of_mu}
\end{equation}
In particular, $\mu(0) = 0$.
These values are used to define, for each~$i \notin C$, preliminary lifting coefficients~$\pi_i \define \max\brackets{h\in \Z : a_i \geq \mu(h)}$.
That is,~$\sum_{i \in C} x_i + \sum_{i \notin C} \pi_ix_i \leq \card{C} - 1$ is valid for~$K^{a,\beta}$, but not necessarily facet defining.
To make these inequalities facet defining, \cite{balas1978facets} has shown that some coefficients~$\pi_i$ need to be increased by~1.
More concretely, for every LCI~\eqref{eq:basic_LCI} defining a facet of~$P^{a,\beta}$, there exists a subset~$S\subseteq \naturalsto{n}\setminus C$ such that~$\alpha_i = \pi_i$ if~$i\notin S$ and~$\alpha_i = \pi_i +1$ if~$i\in S$.
Furthermore,~\cite{nemhauser1994lifted} identified a necessary and sufficient criterion for these sets via a concept called independence.
A set~$S\subseteq N\setminus C$ is called \emph{independent} if for any subset~$Q\subseteq S$ we have 
\begin{equation}
    \sum_{i\in Q} a_i > \mu\left(\sum_{i\in Q} (\pi_i +1)\right) - \Delta(C),
    \label{eq:def_indep_set}
\end{equation} 
where~$\Delta(C)$ denotes the difference between the weight of the cover and the capacity of the knapsack.
An independent set~$S$ is called \emph{maximal} if there is no other independent set containing~$S$.
The characterization of~\cite{nemhauser1994lifted} reads then as follows:
\begin{theorem}[\cite{nemhauser1994lifted}]
  Let~$a \in \Z_{+}^n$ and let~$\beta$ be an integer satisfying~$\beta \geq a_i$ for all~$i \in \naturalsto{n}$.
  Then,
    \begin{equation}
        \sum_{i\in C} x_i + \sum_{i\in S} (\pi_i + 1) x_i + \sum_{i\notin C\cup S} \pi_i x_i \leq |C| - 1
        \label{eq:cover_full_ineq}
    \end{equation}
    defines a facet of~$P^{a, \beta}$ if and only if~$C$ is a minimal cover and~$S$ a maximal independent set.
\end{theorem}

\subsection{Proof of Theorem~\ref{thm:complexitySparse}}
\label{sec:proofComplexity}

\noindent
Before proving Theorem~\ref{thm:complexitySparse}, we note that sparsity of a knapsack does not rule out the existence of super-polynomially many minimal covers as demonstrated by the following example.
\begin{example}
  Let~$n$ and~$k$ be positive integers with~$k\leq n$.
  The knapsack~$\sum_{i=1}^{n-1} x_i + 2 x_n \leq k$ has sparsity~$2$
  and two types of minimal covers: selecting~$k+1$ elements of weight~$1$ or selecting~$k-1$ elements of weight~$1$ and the element of weight~$2$.
  This means that there are~$\binom{n-1}{k+1} + \binom{n-1}{k-1}$ possible minimal covers.
\end{example}
As the example illustrates, it makes sense not to consider minimal covers independently, but to group them into families of similarly structured covers.
This way, we might be able to reduce an exponential number of covers to polynomially many families of covers, and the separation problem can be solved within each family independently.
To prove Theorem~\ref{thm:complexitySparse}, we will follow this idea.
It will therefore be convenient to group variables~$x_i$ by their knapsack coefficient~$a_i$, to which we refer to in the following as \emph{weights}.
Let~$W = \{a_i : i \in \naturalsto{n}\}$ be the set of distinct weights and let~$\sparsity = \card{W}$.
Assume~$W = \{w_1,\dots,w_\sparsity\}$ with~$w_1 < w_2 < \dots < w_\sparsity$, and define, for~$j \in [\sparsity]$, $W_j = \{i \in \naturalsto{n} : a_i = w_j\}$.
The knapsack inequality can then be rewritten as 
\[
  \sum_{j=1}^\sparsity w_j  x(W_j) \leq \beta.
\]
Based on this representation, we define an equivalence relation~$\sim$ on the power set of~$\naturalsto{n}$ as follows.
For two sets~$A, A' \subseteq \naturalsto{n}$, we say~$A\sim A'$ if and only if~$\card{A\cap W_j} = \card{A'\cap W_j}$ for all~$j \in [\sparsity]$.
We collect some basic facts about this equivalence relation.
\begin{observation}
  Let~$a \in \Z_+^n$, $\beta \in \Z_+$ such that~$a$ is~$\sparsity$-sparse, and let~$C$ be a minimal cover of~$K^{a,\beta}$.
  \begin{enumerate}
  \item If~$C' \subseteq \naturalsto{n}$ satisfies~$C \sim C'$, then~$C'$ is a minimal cover.
  \item Let~$\gamma$ be a permutation of~$\naturalsto{n}$ such that~$\gamma(W_j) = W_j$ for all~$j \in [\sparsity]$.
    Then,~$\gamma(C)$ is a minimal cover of~$K^{a,\beta}$ with corresponding cover inequality
    \begin{equation}
      \label{eq:permuted_cover_ineq}
      \sum_{i\in C} x_{\gamma(i)} \leq \card{C} - 1.
    \end{equation}
  \end{enumerate}
\end{observation}
Based on this observation, we can solve the separation problem of minimal cover inequalities for a given vector~$\hat{x}$ as follows.
We iterate over all equivalence classes~$\mathcal{C}$ of minimal covers, and we look for a minimal cover~$C^{\max} \in \mathcal{C}$ whose left-hand side is maximal w.r.t.~$\hat{x}$, i.e., $\hat{x}(C) \leq \hat{x}(C^{\max})$ for all~$C\in\mathcal{C}$.
Since the right-hand side of all minimal cover inequalities for covers in~$\mathcal{C}$ is the same, a violated inequality within class~$\mathcal{C}$ exists if and only if the inequality for~$C^{\max}$ is violated.
This idea naturally extends to the LCIs:

In this case, for a given minimal cover~$C$ and corresponding maximal independent set~$S$, an equivalence class is defined as~$\mathcal{M}(C,S)$ consisting of all pairs~$(C',S') \in \naturalsto{n} \times \naturalsto{n}$ with~$S' \cap C' = \emptyset$, $C' \sim C$, and~$S' \sim S$.
Then, there exists a violated LCI within the class~$\mathcal{M}(C,S)$ for the point~$\hat{x}$ if and only if the inequality corresponding to the following pair of cover and independent set is violated:
\begin{equation*}
    \left(C , S\right)^{\max} \define \underset{(C', S')\in \mathcal{M}(C, S)}{\argmax} \brackets{
    \sum_{i\in C'} \hat{x}_i + \sum_{i\in S'} (\pi_i + 1) \hat{x}_i + \sum_{i\notin C'\cup S'} \pi_i \hat{x}_i 
    }.
\end{equation*}
We can obtain the pair~$(C, S)^{\max}$ by independently inspecting the weight classes~$W_j$, $j \in \naturalsto{\sparsity}$, as follows:
\begin{enumerate}
  \item Set~$S\cap W_j$ to be the~$\card{S\cap W_j}$ largest values of~$\brackets{\hat{x}_l : l\in W_j}$.
  \item Depending on the value of~$\pi_j$:
  \begin{enumerate}
    \item If~$\pi_j \geq 1$, set~$C\cap W_j$ to be the indices of the~$\card{C\cap W_j}$ smallest values of~$\brackets{\hat{x}_l : l\in W_j\setminus S}$.
    \item If~$\pi_j = 0$, set~$C\cap W_j$ to be the indices of the~$\card{C\cap W_j}$ largest values of~$\brackets{\hat{x}_l : l\in W_j\setminus S}$.
  \end{enumerate}
\end{enumerate}

Observe that we can write the inequality explicitly in the special case if all weight classes are sorted.
Formally, when~$W_j$ is given by~$\brackets{i_1, \dots, i_{\card{W_j}}}$, the point~$\hat{x}$ is sorted if~$\hat{x}_{i_1}\leq \dots \leq \hat{x}_{i_{\card{W_j}}}$.
We again have to make the distinction
\begin{align*}
  \nu_j(i) &= \begin{cases}
    1         & \text{if } 1 \leq i \leq \card{W_j\cap C} \\
    \pi_j     & \text{if } \card{W_j\cap C}+1 \leq i \leq \card{W_j} - \card{W_j\cap S} \\
    \pi_j + 1 & \text{if } \card{W_j} - \card{W_j\cap S}+1 \leq i \leq {\card{W_j}}
  \end{cases}, && \text{if } \pi_j \geq 1 \\
  \nu_j(i) &= \begin{cases}
    0         & \text{if } 1 \leq i \leq \card{W_j} - \card{W_j\cap (C\cup S)} \\
    1         & \text{if } \card{W_j} - \card{W_j\cap (C\cup S)}+1 \leq i \leq \card{W_j}
  \end{cases}, && \text{if } \pi_j =0 
\end{align*}
to write a most violated cut in~$\mathcal{M}(C, S)$ as
\begin{equation}
  \sum_{j=1}^\sparsity \sum_{i=1}^{\card{W_j}} \nu_j(i)\cdot x_{j_i} \leq \card{C} -1.\label{eq:strongest-cut}
\end{equation}
Based on the representative~$(C,S)^{\max}$ of an equivalence class, we can prove Theorem~\ref{thm:complexitySparse}.
\begin{proof}[Proof of Theorem~\ref{thm:complexitySparse}]
  In a first step, we observe that there are only polynomially many equivalence classes~$\mathcal{M}(C,S)$ that we can enumerate explicitly.
  Indeed, an equivalence class~$\mathcal{C}$ is fully determined by the number of elements in each weight class~$c_j = \card{C\cap W_j}$ for any~$C\in\mathcal{C}$ and~$j\in\naturalsto{\sparsity}$.
  Since~$c_j\leq n$ for all~${j\in\naturalsto{\sparsity}}$, every minimal cover is represented by an element of~$\naturalsto{n}^{\sparsity}$.
  Such a set~$C$ corresponds to a cover if and only if~$\sum_{j=1}^\sparsity w_j c_j > \beta$.
  Similarly, the cover will be minimal if and only if we also have~$\sum_{j=1}^\sparsity w_j\cdot c_j - w_{j^*} \leq \beta$ where~$j^*\in \argmin\brackets{j\in\naturalsto{\sparsity} : c_j > 0}$.
  Consequently, we can exhaustively enumerate all families of minimal covers in~$\Oh{\sparsity n^\sparsity}$ time.
  In fact, we can lower this bound to~$\Oh{\sparsity n^{\sparsity -1}}$ because for any given~$c_1, \dots, c_{\sparsity -1}$, there exists a unique~$c_{\sparsity}$, if feasible, such that the corresponding set is a minimal cover, namely~$c_{\sparsity}= \ceil{\nicefrac{\left(\beta+1 - \sum_{j=1}^{\sparsity -1} c_j w_j\right)}{w_{\sparsity}}}$.
  Since, for each chosen cover, the value~$\mu(h + 1)$ can be computed in constant time from~$\mu(h)$, the values~$\mu(h)$, $h\leq n$, can be determined in~$\Oh{n}$ time after which we can use~$\mu$ as a constant-time lookup table.
  The families of possible sets~$S$ for a given~$C$ are also uniquely defined by the cardinality of~$S\cap W_j$.
  As such, we can again list all potential independent sets in~$\Oh{\sparsity n^\sparsity}$.
  Note in particular that evaluating a set~$Q\subseteq S$ using the formula~\eqref{eq:def_indep_set} becomes
  \begin{equation*}
    \sum_{j=1}^\sparsity q_j w_j > \mu\left(\sum_{j=1}^\sparsity q_j\cdot(\pi_j +1)\right) - \Delta(C)
  \end{equation*} 
  with denoting~$q_j = |Q\cap W_j|$ for all~$j\in\naturalsto{\sparsity}$, and can now be done in~$\Oh{\sparsity}$ via the lookup table for~$\mu(\cdot)$.
  Additionally, for~$S$ to be independent all subsets of~$S$ must also be independent.
  This can be checked dynamically in~$\Oh{n^\sparsity}$ time by enumerating all possible~$S$ increasingly with respect to~$\card{S}$ and saving the verdict for all sets in some large table.
  Then the set~$S$ is independent if it satisfies~\eqref{eq:def_indep_set} and all~$Q\subsetneq S$ where~$\card{Q}=\card{S}-1$, of which there are at most~$\sparsity$ non-equivalent, are also independent sets.

  To conclude the proof, it is sufficient to find, for each equivalence class~$\mathcal{M}(C,S)$ a maximal representative~$(C,S)^{\max}$ as defined above.
  This can be achieved by sorting the point~$\hat{x}$ to be separated on each of the weight classes, which takes~$\Oh{n\log{n}}$ time and evaluating~\eqref{eq:strongest-cut}.
  The whole separation routine can thus be implemented in~$\Oh{n\log{n} + \sparsity n^{\sparsity-1}\cdot \sparsity n^\sparsity\cdot n} = \Oh{\sparsity^2 n^{2\sparsity}}$ time.
\end{proof}

\section{Polyhedral Models for Separation Algorithms}
\label{sec:Sorting-network}

\noindent
In the previous section, we have seen that LCIs for sparse knapsacks can be separated in polynomial time.
A potential downside of this approach, however, is that implications of LCIs cannot directly be observed by an integer programming solver, but must be learned via separation.
In particular, the first LP relaxation to be solved does not contain any LCI.
It might be possible though to define a single inequality that models implications of an entire equivalence class of LCIs as shown by the following example, which is inspired by an approach of Riise et al.~\cite{riise2016recursive}. 
\begin{example}
    Let us consider the knapsack 
    \begin{equation*}
        x_1 + x_2 + x_3 + x_4 + x_5 + 2\cdot (x_6 + x_7 + x_8 + x_9 + x_{10}) \leq 10.
    \end{equation*}
    We can represent families of equivalent covers by adding the binary variables~$z_{i,j}$ that are~$1$ if and only if~$j$ elements of weight~$i$ are selected.
    All cover inequalities where the cover has three elements of weight~$1$ and four elements of weight~$2$ can then be represented by~$z_{1,3} + z_{2,4} \leq 1$.
\end{example}
The approach developed in this section is inspired by this idea, but our goal is to avoid the introduction of auxiliary \emph{integer} variables.
On a high level, for a given equivalence class of LCIs, we will introduce auxiliary \emph{continuous} variables~$y \in \R^m$ and a polyhedron~$P \subseteq \R^n \times \R^m$ such that a point~$(x,y) \in \R^n \times \R^m$ is contained in~$P$ if and only if~$x$ satisfies \emph{all} LCIs from the given equivalence class.
Our hope is that, if~$P$ can be described by few inequalities, we can add these inequalities to an integer program and avoid the separation algorithm of LCIs presented in the previous section.
We refer to the polyhedron~$P$ as a \emph{separation polyhedron}.
\begin{remark}
    For an equivalence class~$\mathcal{M}(C, S)$ of covers and corresponding independent sets, let~$\bar{P}$ be the set of all~$x\in\cube{n}$ that satisfy all equivalent LCIs to Equation~\eqref{eq:cover_full_ineq}.
    In other words,
    \begin{equation}
        \bar{P} \define \bigcap_{(C', S')\in\mathcal{M}(C, S)} \brackets{
            x\in\cube{n}: 
            \sum_{i\in S'} (\pi_i +1) x_i 
            + \sum_{i\notin C'\cup S'} \pi_i x_i
            + \sum_{i\in C'} x_i
            \leq \card{C} - 1
        }.
        \label{eq:polytope-all-cuts}
    \end{equation}
    If we do not introduce auxiliary variables, the separation polyhedron~$P$ will be given by~$\bar{P}$, and thus requires potentially exponentially many inequalities in an outer description.
    By introducing auxiliary variables though, we define a so-called extended formulation of~$\bar{P}$, which might allow to reduce the number of inequalities needed in a description drastically~\cite{conforti2010extended}.
\end{remark}

Recall that the main insight of the separation algorithm for LCIs was that we can apply the LCI that dominates its equivalence class if we sort certain variables by their value in a solution~$\hat{x} \in \R^n$.
A naive approach to achieve our goal is thus to look for a polyhedron~$P$ that models
\[
  \mathcal{X}_n \define \{(x,y) \in \R^n \times \R^n : \text{$y$ is a sorted copy of~$x$}\}
\]
and define a separation polyhedron as
\begin{equation}
  \label{eq:sortingIdea}
  \{ (x,y) \in P : \text{$y$ satisfies the LCI for~$(C,S)^{\max}$}\}
\end{equation}
for the most violated LCI w.r.t.\ a sorted vector as defined in Section~\ref{sec:proofComplexity}.
This is impossible though as the set~$\mathcal{X}_n$ is not convex in general.
\begin{lemma}
    \label{lemma:non-convex-sorted}
    For any~$n \geq 2$ the set~$\mathcal{X}_n$ is not convex.
\end{lemma}
\begin{proof}
  Let~$(x^1, y^1)$,~$(x^2, y^2) \in \mathcal{X}_n$ be such that~$x^1 = (1, 0, \ldots, 0)$ and~$x^2 = (0, 1, 0, \ldots, 0)$.
  Then, we have~$y^1 = y^2 = (0, \ldots, 0, 1)$.
  For any~$\lambda_1, \lambda_2 \in (0, 1)$ with~$\lambda_1 + \lambda_2 = 1$, we have that~$(x^3, y^3) = \lambda_1 (x^1, y^1) + \lambda_2 (x^2, y^2)$ belongs to the convex hull of~$\mathcal{X}_n$.
  However~$y^3 = (0, \ldots, 0, 1)$ is not a sorted version of~$x^3 = (\lambda_1, \lambda_2, 0, \dots, 0)$.
  Hence,~$\mathcal{X}_n$ is not convex.
\end{proof}

Nevertheless, the method we present carries the same core idea, but we need to refine the sorting mechanism.
To this end, we will make use of so-called sorting networks that we discuss in the next section.
Afterward, we will show how sorting networks can be used to define a sorting polyhedron for an equivalence class of LCIs that only requires~$\Oh{n \log{n}}$ inequalities.

\subsection{Sorting Networks}

\noindent
Despite the existence of efficient sorting algorithms, sorting networks have been introduced to offer strong alternatives in the context of systems that can process several instructions at the same time.
We provide a formal definition of sorting networks next, following the description of~\cite{cormen2001introduction}.

Sorting networks are a special case of so-called comparison networks.
Let~$n$ and~$K$ be positive integers.
A \emph{$(n,K)$-comparison network} consists of~$n$ so-called wires and~$K$ so-called comparators, which are pairs of wires~$(i_k,j_k)$, $k \in \naturalsto{K}$, such that~$i_k < j_k$.
Comparison networks can be illustrated by drawing wires as horizontal lines (labeled~$1,\dots,n$ from top to bottom) and comparators~$(i_k,j_k)$ as vertical lines connecting the two wires~$i_k$ and~$j_k$, see Figure~\ref{fig:sorting_network}.
We assume that vertical lines are sorted based on their index~$k$, i.e., if~$k,k' \in \naturalsto{K}$ satisfies~$k < k'$, then comparator~$k$ is drawn to the left of comparator~$k'$.

Given a vector~$\hat{x} \in \R^n$, a comparison network can be used to partially sort the entries of~$\hat{x}$.
To this end, we introduce a partial sorting function~$\phi_{\hat{x}}(l, k)$ for~$l \in \naturalsto{n}$ and~$k \in \naturalsto[0]{K}$ as follows:
\[
  \phi_{\hat{x}}(l, k)
  =
  \begin{cases}
    l, & \text{if } k = 0,\\
    \phi_{\hat{x}}(l, k-1), & \text{if } k \geq 1 \text{ and } \phi_{\hat{x}}(l, k-1) \notin \{i_k,j_k\},\\
    i_k, & \text{if } k \geq 1,\; \phi_{\hat{x}}(l, k-1) \in \{i_k,j_k\} \text{ and for } l' \in \naturalsto{n}\\
       & \text{such that } \{i_k,j_k\} = \{\phi_{\hat{x}}(k-1,l), \phi_{\hat{x}}(k-1,l')\}\\
       & \text{and } \hat{x}_{l'} \geq \hat{x}_{l},\\
    j_k, & \text{otherwise}.
  \end{cases}
\]
The function~$\phi_{\hat{x}}$ can be interpreted as follows.
We assign each entry~$\hat{x}_l$, $l \in \naturalsto{n}$, to the left end of wire~$l$, which is captured by~$\phi_{\hat{x}}(\cdot, 0)$.
Then, the entries travel along the wires from left to right at the same speed, where we interpret index~$k \in \naturalsto{K}$ as a time step.
When two entries reach a comparator~$(i_k,j_k)$ at time~$k$, the values assigned to wires~$i_k$ and~$j_k$ are compared.
If the value assigned to wire~$j_k$ is at most the value assigned to wire~$i_k$, the value assignment of both wires is swapped.
Otherwise, the entries travel further along their previous wires.
The value~$\phi_{\hat{x}}(l, k)$ can thus be interpreted as the position of entry~$\hat{x}_l$ in a reordered vector after~$k$ comparisons.
In particular, $\phi_{\hat{x}}(\cdot, k)$ is a permutation of~$\naturalsto{n}$.
\begin{figure}[t]
    \centering
    \begin{tikzpicture}[scale=0.45]

    \node (l1s) at (-1, 6)  {};
    \node (l1e) at (15, 6)  {};
    \node (l2s) at (-1, 4)  {};
    \node (l2e) at (15, 4)  {};
    \node (l3s) at (-1, 2)  {};
    \node (l3e) at (15, 2)  {};
    \node (l4s) at (-1, 0)  {};
    \node (l4e) at (15, 0)  {};

    \draw[->] (l1s) -- (l1e);
    \draw[->] (l2s) -- (l2e);
    \draw[->] (l3s) -- (l3e);
    \draw[->] (l4s) -- (l4e);

    \node (s11) at ($(l1s)!2/12!(l1e)$)  {};
    \node (s12) at ($(l2s)!2/12!(l2e)$)  {};
    \node (s21) at ($(l3s)!4/12!(l3e)$)  {};
    \node (s22) at ($(l4s)!4/12!(l4e)$)  {};
    \node (s31) at ($(l1s)!6/12!(l1e)$)  {};
    \node (s32) at ($(l3s)!6/12!(l3e)$)  {};
    \node (s41) at ($(l2s)!8/12!(l2e)$)  {};
    \node (s42) at ($(l4s)!8/12!(l4e)$)  {};
    \node (s51) at ($(l2s)!10/12!(l2e)$) {};
    \node (s52) at ($(l3s)!10/12!(l3e)$) {};

    \draw[ultra thick] (s11) -- (s12);
    \draw[ultra thick] (s21) -- (s22);
    \draw[ultra thick] (s31) -- (s32);
    \draw[ultra thick] (s41) -- (s42);
    \draw[ultra thick] (s51) -- (s52);

    \draw[decorate,decoration=zigzag] (l2s)
        -- (s12) 
        -- (s11) 
        -- (s31) 
        -- (s32) 
        -- (s52) 
        -- (s51) 
        -- (l2e);

    \foreach \x in {s11, s12, s21, s22, s31, s32, s41, s42, s51, s52}
        \filldraw[black] (\x) circle[radius=0.15];

    \draw ($(l1s)!1/12!(l1e)$)  node[above] {$4$};
    \draw ($(l2s)!1/12!(l2e)$)  node[above] {$2$};
    \draw ($(l3s)!1/12!(l3e)$)  node[above] {$1$};
    \draw ($(l4s)!1/12!(l4e)$)  node[above] {$3$};

    \draw ($(l1s)!3/12!(l1e)$)  node[above] {$2$};
    \draw ($(l2s)!3/12!(l2e)$)  node[above] {$4$};
    \draw ($(l3s)!3/12!(l3e)$)  node[above] {$1$};
    \draw ($(l4s)!3/12!(l4e)$)  node[above] {$3$};

    \draw ($(l1s)!5/12!(l1e)$)  node[above] {$2$};
    \draw ($(l2s)!5/12!(l2e)$)  node[above] {$4$};
    \draw ($(l3s)!5/12!(l3e)$)  node[above] {$1$};
    \draw ($(l4s)!5/12!(l4e)$)  node[above] {$3$};

    \draw ($(l1s)!7/12!(l1e)$)  node[above] {$1$};
    \draw ($(l2s)!7/12!(l2e)$)  node[above] {$4$};
    \draw ($(l3s)!7/12!(l3e)$)  node[above] {$2$};
    \draw ($(l4s)!7/12!(l4e)$)  node[above] {$3$};

    \draw ($(l1s)!9/12!(l1e)$)  node[above] {$1$};
    \draw ($(l2s)!9/12!(l2e)$)  node[above] {$3$};
    \draw ($(l3s)!9/12!(l3e)$)  node[above] {$2$};
    \draw ($(l4s)!9/12!(l4e)$)  node[above] {$4$};
    
    \draw ($(l1s)!11/12!(l1e)$)  node[above] {$1$};
    \draw ($(l2s)!11/12!(l2e)$)  node[above] {$2$};
    \draw ($(l3s)!11/12!(l3e)$)  node[above] {$3$};
    \draw ($(l4s)!11/12!(l4e)$)  node[above] {$4$};

\end{tikzpicture}
    \caption{Example of a sorting network.}
    \label{fig:sorting_network}
\end{figure}
\begin{example}
    Figure~\ref{fig:sorting_network} shows a sorting network on 4 variables.
    $G$ is composed of the five successive comparisons:~$(1, 2)$, $(3, 4)$, $(1, 3)$, $(2, 4)$, and~$(2, 3)$.
    Here the starting vector~$\hat{x}$ is~$(4, 2, 1, 3)$ and thus the output is~$(1, 2, 3, 4)$.
    The zigzagging path highlights the positions of the value~$2$.
    We then have~$\phi(2, 0)=\phi(2, 5) =2$,~$\phi(2, 1) = \phi(2, 2) = 1$ and~$\phi(2, 3) = \phi(2, 4) = 3$.
\end{example}
In the following, we denote a comparison network by~$G = \brackets{(i_k, j_k) : k \in \naturalsto{K}}$.
A comparison network is called a \emph{sorting network} if, for every~$\hat{x} \in \R^n$, the corresponding function~$\phi_{\hat{x}}(\cdot, K)$ is a permutation of~$\naturalsto{n}$ that sorts the entries of~$\hat{x}$ non-increasingly.
Small sorting networks exist for all positive integers~$n$.
The main benefit of this method is that two consecutive comparisons that are on a disjoint pair of wires can be done in parallel, in the same time step~$k$.
This allows for even more compact sorting networks where some comparisons can be done simultaneously, reducing the number of comparison steps from~$K$ to~$K'$.
\begin{proposition}[\cite{cormen2001introduction}]
    There exists sorting networks that sort a vector~$\hat{x}\in\cube{n}$ where~$K'=\Oh{\log[]{n}^2}$ using~$\Oh{n\log[]{n}}$ comparisons.
\end{proposition}
A recursive description for building such a sorting network in~$\Oh{n\log[]{n}}$ time is given in~\cite{cormen2001introduction}.
Moreover, \cite{cormen2001introduction} gives a list of some sorting networks that are straightforward to build, such as the one imitating insertion sort.
However, for the remainder of the chapter, we will only describe techniques and polytopes based on sorting networks with only one comparison per step.
This is because the adaptation of the proofs and constructions for the parallelized version are rather intuitive but heavy on notation.

\subsection{The Sorting Network Polytope}
\label{sec:sorting-network-polytope}

\noindent
Equipped with the concept of sorting networks, we will now derive a sorting polyhedron for fixed vectors~$\hat{x} \in [0,1]^n$, which is based on the idea presented in~\eqref{eq:sortingIdea}.
Later, we will discuss how the assumption that~$\hat{x}$ is fixed can be dropped to make use of it in modeling the separation problem of LCIs.
The construction of the sorting polyhedron is based on~\cite{goemans2015permutahedron}.

Let~$G = \brackets{(i_k, j_k) : k \in \naturalsto{K}}$ be a sorting network for~$n$-dimensional vectors.
We introduce auxiliary variables~$x^k\in\cube{n}$, $k \in \naturalsto[0]{K}$, which shall correspond to the partially sorted vectors after~$k$ steps.
The comparisons~$(i_k,j_k)$ then induce the following constraints:
\begin{subequations}
    \label{eq:explicit-sorting-network}
    \begin{alignat}{7}
        x^{k-1}_{i_k} && && -x^k_{i_k} &&   && \geq & \ 0, && && k\in\naturalsto{K}, \label{eq:alpha-eq-1} \\
        && x^{k-1}_{j_k} && -x^k_{i_k} &&   && \geq & \ 0, && && k\in\naturalsto{K}, \label{eq:alpha-swap-1} \\
        -x^{k-1}_{i_k} && && && +x^k_{j_k}  && \geq & \ 0, && && k\in\naturalsto{K}, \label{eq:alpha-swap-2} \\   
        && -x^{k-1}_{j_k} && && +x^k_{j_k}  && \geq & \ 0, && && k\in\naturalsto{K}, \label{eq:alpha-eq-2} \\
        -x^{k-1}_{i_k} && -x^{k-1}_{j_k} && +x^k_{i_k} && +x^k_{j_k} && = & \ 0, && && k\in\naturalsto{K}, \label{eq:beta-k} \\
        && -x^{k-1}_l && && +x^k_l  && = & \ 0 , \quad && l \in \naturalsto{n}\setminus\brackets{i, j}, \quad && k\in\naturalsto{K},  \label{eq:delta-k-l}\\
        && && && -x^k_l && \geq & -1 , \quad && l \in \naturalsto{n}, && k\in\naturalsto{K}, \label{eq:lambda} \\
        && && && x^k_l && \geq & 0 , \quad && l \in \naturalsto{n}, && k\in\naturalsto{K}, \label{eq:nonnegative} \\
        && && && x^0_l && = & \ \hat{x}_l, \quad && l \in \naturalsto{n}. && \label{eq:delta-0-l}
    \end{alignat}
\end{subequations}
We refer to the polytope defined by these constraints as~$P(G, \hat{x})$.
Observe that Constraints~\eqref{eq:delta-k-l} and~\eqref{eq:delta-0-l} force all entries that are not compared at step~$k$ to be copied over in the next step.
Constraints~\eqref{eq:alpha-eq-1} and~\eqref{eq:alpha-swap-1} ensure that that~$x^k_{i_k} \leq \min\brackets{x^{k-1}_{i_k}, x^{k-1}_{j_k}}$ and, in a similar way, Constraints~\eqref{eq:alpha-eq-2} and~\eqref{eq:alpha-swap-2} ensure that~$x^k_{j_k} \geq\max\brackets{x^{k-1}_{i_k}, x^{k-1}_{j_k}}$.
Together with Constraint~\eqref{eq:beta-k}, the first five constraints of System~\eqref{eq:explicit-sorting-network} are equivalent to
\begin{align*}
  x^k_{i_k} &= \min\brackets{x^{k-1}_{i_k}, x^{k-1}_{j_k}} - \epsilon \\
  x^k_{j_k} &= \max\brackets{x^{k-1}_{i_k}, x^{k-1}_{j_k}} + \epsilon
\end{align*}
for some~$\epsilon \geq 0$.
That is, solutions adhering to these inequalities do not necessarily correspond to reorderings of the initial vector~$\hat{x}$.
In practice, however, it is enough for the sorted copy of~$x^{k-1}_{i_k}, x^{k-1}_{j_k}$ to be part of the feasible values of~$x^k_{i_k}, x^k_{j_k}$.
\begin{lemma}
    \label{lemma:sorting-network-works}
    Let~$G$ be a sorting network, let~$\hat{x}\in\cube{n}$, and~$P(G, \hat{x})$ the sorting network polytope as in~\eqref{eq:explicit-sorting-network}.
    Then there exists a feasible point~$(\tilde{x}^0, \ldots, \tilde{x}^K)\in P(G, x)$ such that~$\hat{x}_l = \tilde{x}^k_{\phi(l, k)}$ for all~$l\in\naturalsto{n}$, $k\in\naturalsto[0]{K}$.
\end{lemma}
\begin{proof}
    We observe that System~\eqref{eq:explicit-sorting-network} has a block structure that is induced by the indices~$k \in \naturalsto{K}$ and two blocks overlap if they have consecutive indices.
    The assertion then follows by a standard inductive argument that exploits that the pair~$\tilde{x}^k_{i_k}, \tilde{x}^k_{j_k}$ satisfies~\eqref{eq:alpha-eq-1} to~\eqref{eq:beta-k}.
\end{proof}

Next, we discuss how System~\eqref{eq:explicit-sorting-network} can be used to replace the exponential amount of inequalities defining~$\bar{P}$ in~\eqref{eq:polytope-all-cuts}.
Recall that our goal is to determine if a point~$\hat{x}$ lies in~$\bar{P}$ or not.
To that end, we have seen in Section~\ref{sec:proofComplexity} that~$\hat{x}\in\bar{P}$ is equivalent to~$\hat{x}$ satisfying Inequality~\eqref{eq:strongest-cut} which requires a permutation sorting the values of~$\hat{x}$ within each weight class.
We emulate this sorting of variables through sorting networks.
Let $G_1, \dots, G_\sparsity$ be sorting networks for the weight classes~$W_1, \dots, W_\sparsity$.
By extending with trivial layers if needed, we can assume that they all use~$K$ steps.
Let~$P_j(G_j, \hat{x})$ be the corresponding comparison polytope for each~$j\in\naturalsto{\sparsity}$ and denote
\begin{equation}
    P\define\brackets{
        (x^0, \dots, x^K)\in \bigotimes_{k=0}^K\cube{n} : 
        (x^0, \dots, x^K)\in \bigcap_{j=1}^\sparsity P_j(G_j, \hat{x})
    }.
    \label{eq:all-sorted-networks}
\end{equation}
In the following, we show that using the polyhedron~$P$ as defined in~\eqref{eq:all-sorted-networks} combined with the idea of~\eqref{eq:sortingIdea} indeed yields an extended formulation of~$\bar{P}$.
The main ingredient of the proof will be the insight that, for a given vector~$\hat{x}$, the left-hand side value of the LCI for~$(C,S)^{\max}$ is the same as the minimal value of the left-hand side that is achievable over~$P$ w.r.t.\ component~$x^K$.
That is, because the sorted version of~$\hat{x}$ is contained in the~$K$-th component of~$P$, $\hat{x}$ violates an LCI if and only if~$x^K$ violates the LCI for~\eqref{eq:strongest-cut}.
Since the different weight classes of the knapsack inequality can be sorted independently, it is sufficient to prove the statement for the different polyhedra~$P_j$ independently.

\begin{proposition}
    \label{thm:sorting-network-dual-primal}
    Let~$G$ be a sorting network on~$n$ variables in~$K$ steps. 
    Let~$\hat{x}\in \cube{n}$ be a fixed input and~$0\leq\myw_1 \leq \ldots \leq \myw_n$ ordered general coefficients.
    Let~$P(G, \hat{x})$ be as in~\eqref{eq:explicit-sorting-network}.
    Let~$\phi(l, k)$ denote the position of the value~$\hat{x}_l$ in~$G$ at step~$k$.
    Then the point~$\left(\tilde{x}^0, \ldots, \tilde{x}^K\right)$
    where~$\tilde{x}^k_{\phi(l , k)} = \hat{x}_l$ is an optimal solution to
    \[
      \min\brackets{\sum_{l=1}^n \myw_lx^K_l : x \in P(G,\hat{x})}.
    \]
\end{proposition}
\begin{proof}

    Using Lemma~\ref{lemma:sorting-network-works}, we know that the point~$\left(\tilde{x}^0, \ldots, \tilde{x}^K\right)$ is a feasible solution to this linear program with objective value~$\sum_{l=1}^{n} \myw_l \hat{x}_{\phi(l, K)}$.
    To prove that it is optimal, we will construct a dual solution with the same objective value~$\sum_{l=1}^{n} \myw_{\psi(l)} \hat{x}_l$, where~$\psi = \phi^{-1}(\cdot, K)$.


    For all~$l\in \naturalsto{n}$ and for all~$ k \in \naturalsto[0]{K-1}$, the variable~$x^k_l$ appears in constraints of~\eqref{eq:explicit-sorting-network} either when~$x^k$ is the output of step~$k$ or the input of step~$k+1$, as well as in~\eqref{eq:lambda} and~\eqref{eq:nonnegative}.
    The type of constraints in which~$x^k_l$ appears depend on the three cases~$l=i_k$,~$l=j_k$ or~$l\notin\brackets{i_k, j_k}$ and the same three cases for the input at~$k+1$, resulting in nine possible dual constraints explicitly written in~\eqref{eq:all-dual-constraints}. 
    We use the shorthand \Up\ if~$l$ is the upper wire of the comparison, \Down\ if it is the lower one and \Neutral\ when~$l$ is not in the current comparison.
    This allows us to write all combinations in the \texttt{AB} format where \texttt{A} and \texttt{B} are~$l$'s position at step~$k$ and~$k+1$, respectively.
    Observe that when~$k=K$ there is no step~$K+1$ to be the input of for~$x^K$ so constraints corresponding to that layer have no \texttt{B} part.

    For each comparison~$(k, \brackets{i, j})\in G$ we get a single dual variable~$\beta^k$ from~\eqref{eq:beta-k},~$n-2$ variables with~$\delta^k_l$ for all~$l\notin \brackets{i, j}$ from~\eqref{eq:delta-k-l} and four non-negative variables~$\myalpha{=}{1}{k}$, $\myalpha{\times}{1}{k}$, $\myalpha{\times}{2}{k}$ and~$\myalpha{=}{2}{k}$ from~\eqref{eq:alpha-eq-1} to~\eqref{eq:alpha-eq-2} respectively.
    Note that although there are no comparisons at the zero-th layer, Equation~\eqref{eq:delta-0-l} behaves similarly to~\eqref{eq:delta-k-l} and as such we use the~$n$ variables~$\delta^0_l$ to represent them.
    Finally, each~$(l, k) \in\naturalsto{n} \times\naturalsto[0]{K}$ induces the non-negative variables~$\lambda^k_l$ and~$\theta^k_l$ from~\eqref{eq:lambda} and~\eqref{eq:nonnegative}, respectively.
    %
    %
    %
    The~$\alpha$ variables are grouped in two pairs~$\alpha^=$ and~$\alpha^\times$ because they correspond to either preserving~$\tilde{x}^k_{i_k}, \tilde{x}^{k}_{j_k}$ on their wires or swapping them.
    On the one hand, if~$\tilde{x}^k_{i_k} = \tilde{x}^{k+1}_{i_k}$, then necessarily~\eqref{eq:alpha-eq-1} and~\eqref{eq:alpha-eq-2} must be tight, which is represented by their values continuing horizontally~$(=)$ in~$G$.
    On the other hand, if~$\tilde{x}^k_{i_k} = \tilde{x}^{k+1}_{j_k}$, then it is~\eqref{eq:alpha-swap-1} and~\eqref{eq:alpha-swap-2} that must be tight, represented by the values exchanging positions~$(\times)$ in~$G$.
    Note that in the case where~$\tilde{x}^k_{i_k} = \tilde{x}^{k}_{j_k}$ we arbitrarily choose to treat it as the case~$\alpha^=$ even though~$\alpha^\times$ could also be active.
    \begin{subequations}
      \label{eq:all-dual-constraints}
      \begin{footnotesize}
        \begin{alignat}{5}
            \delta^{k}_l 
            & -\delta^{k+1}_l 
            && +\theta^k_l-\lambda^k_l
            && \leq 0, \
            && l\in\naturalsto{n}, k<K,
            \label{eq:dual-NN} \tag{\theequation-\Neutral\Neutral} \\
            \delta^{k}_l 
            & -\left(\beta^{k+1} - \myalpha{=}{1}{k+1} + \myalpha{\times}{2}{k+1}\right)
            && +\theta^k_l-\lambda^k_l
            && \leq 0 , \
            && l\in\naturalsto{n}, k<K,
            \label{eq:dual-NU} \tag{\theequation-\Neutral\Up} \\
            \delta^{k}_l 
            & -\left(\beta^{k+1} - \myalpha{\times}{1}{k+1} + \myalpha{=}{2}{k+1}\right)
            && +\theta^k_l-\lambda^k_l
            && \leq 0 , \
            && l\in\naturalsto{n}, k<K,
            \label{eq:dual-ND} \tag{\theequation-\Neutral\Down} \\
            \beta^k - \myalpha{=}{1}{k} - \myalpha{\times}{1}{k}
            & -\delta^{k+1}_l 
            && +\theta^k_l-\lambda^k_l
            && \leq 0 , \
            && l\in\naturalsto{n}, k<K,
            \label{eq:dual-UN} \tag{\theequation-\Up\Neutral} \\
            \beta^k - \myalpha{=}{1}{k} - \myalpha{\times}{1}{k}
            & -\left(\beta^{k+1} - \myalpha{=}{1}{k+1} + \myalpha{\times}{2}{k+1}\right)
            && +\theta^k_l-\lambda^k_l
            && \leq 0 , \
            && l\in\naturalsto{n}, k<K,
            \label{eq:dual-UU} \tag{\theequation-\Up\Up} \\
            \beta^k - \myalpha{=}{1}{k} - \myalpha{\times}{1}{k}
            & -\left(\beta^{k+1} - \myalpha{\times}{1}{k+1} + \myalpha{=}{2}{k+1}\right)
            && +\theta^k_l-\lambda^k_l
            && \leq 0 , \
            && l\in\naturalsto{n}, k<K,
            \label{eq:dual-UD} \tag{\theequation-\Up\Down} \\
            \beta^k + \myalpha{\times}{2}{k} + \myalpha{=}{2}{k}
            & -\delta^{k+1}_l 
            && +\theta^k_l-\lambda^k_l
            && \leq 0 , \
            && l\in\naturalsto{n}, k<K,
            \label{eq:dual-DN} \tag{\theequation-\Down\Neutral} \\
            \beta^k + \myalpha{\times}{2}{k} + \myalpha{=}{2}{k}
            & -\left(\beta^{k+1} - \myalpha{=}{1}{k+1} + \myalpha{\times}{2}{k+1}\right)
            && +\theta^k_l-\lambda^k_l
            && \leq 0 , \
            && l\in\naturalsto{n}, k<K,
            \label{eq:dual-DU} \tag{\theequation-\Down\Up} \\
            \beta^k + \myalpha{\times}{2}{k} + \myalpha{=}{2}{k}
            & -\left(\beta^{k+1} - \myalpha{\times}{1}{k+1} + \myalpha{=}{2}{k+1}\right)
            && +\theta^k_l-\lambda^k_l
            && \leq 0 , \
            && l\in\naturalsto{n}, k<K,
            \label{eq:dual-DD} \tag{\theequation-\Down\Down} \\
            \delta^{K}_l
            &
            && +\theta^K_l -\lambda^K_l 
            && \leq \myw_l, \
            && l\in\naturalsto{n},
            \label{eq:dual-NK} \tag{\theequation-\Neutral -} \\
            \beta^K - \myalpha{=}{1}{K} - \myalpha{\times}{1}{K}
            &
            && +\theta^K_l -\lambda^K_l
            && \leq \myw_l , \
            && l\in\naturalsto{n},
            \label{eq:dual-UK} \tag{\theequation-\Up -} \\
            \beta^K + \myalpha{\times}{2}{K} + \myalpha{=}{2}{K}
            &
            && +\theta^K_l -\lambda^K_l
            && \leq \myw_l , \
            && l\in\naturalsto{n}.
            \label{eq:dual-DK} \tag{\theequation-\Down -}
        \end{alignat}
      \end{footnotesize}
    \end{subequations}

    The dual objective function is~$\sum_{l=1}^n \hat{x}_l \cdot \delta^0_l -\sum_{l=1}^n \sum_{k=0}^K \lambda^k_l$.
    Since~$\tilde{x}^K_l = \hat{x}_{\psi(l)}$, we want to set all~$\delta^0_{l} = \myw_{\psi(l)}$ as well as~$\lambda^k_l = 0$.
    Observe that the dual objective function reduces then to
    \begin{equation}
        \sum_{l=1}^n \hat{x}_l \cdot \myw_{\psi(l)} = 
        \sum_{l=1}^n \hat{x}_{\phi(l, K)} \cdot \myw_l = 
        \sum_{l=1}^n \tilde{x}^K_l \cdot \myw_l. 
    \end{equation}
    More generally, we can safely choose to set all~$\theta^k_l = 0$ for any pair~$l\in\naturalsto{n}, k\in\naturalsto[0]{K}$ since they have no contribution in the objective function and only make the constraint tighter if not set to zero.
    For every~$1\leq k\leq K$ and every~$l\in \naturalsto{n}$, assuming the comparison at step~$k$ is~$(k, \brackets{i_k, j_k })$, we can construct the dual variables ($\delta^k_l$ or~$\beta^k$ and the~$\myalpha{}{}{k}$) by observing whether~$\phi(l, k) \in \brackets{i_k, j_k}$ as well as whether~$\phi(l, k) = \phi(l, k-1)$ or not.
    \begin{enumerate}[label=(\alph*)]
        \item\label{case:(-)} If~$\phi(l, k) \notin \brackets{i_k, j_k}$, then~$\phi(l, k) = \phi(l, k-1)$ and we are in the \texttt{N} situation and we set~$\delta^k_{\phi(l, k)} = \myw_{\phi(l, K)}$.
        
        \item\label{case:(=)} If~$\phi(l, k) \in \brackets{i_k,j_k}$ and~$\phi(l, k) = \phi(l, k-1)$, we can set~$\myalpha{\times}{1}{k} = \myalpha{\times}{2}{k} = 0$.
        Let~$l'$ be the other wire such that~$\brackets{\phi(l, k), \phi(l', k)} = \brackets{i_k, j_k}$.
        Then~$\phi(l', k) = \phi(l', k-1)$ as well.
        If we have~$\phi(l, k) <\phi(l', k)$, then it follows that~$\tilde{x}^k_{\phi(l, k)} \leq \tilde{x}^k_{\phi(l', k)}$ and, since G is a sorting network,~$\phi(l, K) < \phi(l', K)$ in the end.
        Therefore~$\myw_{\phi(l, K)} \leq \myw_{\phi(l', K)}$.
        We can then choose~$\beta^k = \nicefrac{\left(\myw_{\phi(l', K)} + \myw_{\phi(l, K)}\right)}{2}$.
        This allows for~$\myalpha{=}{1}{k} = \beta^k - \myw_{\phi(l, K)}\geq 0$ and~$\myalpha{=}{2}{k} = \myw_{\phi(l', K)} - \beta^k\geq 0$.
        If we have~$\phi(l, k) > \phi(l', k)$, by a symmetric argument we need to change~$\myalpha{=}{1}{k} = \beta^k - \myw_{\phi(l', K)}\geq 0$ and~$\myalpha{=}{2}{k} = \myw_{\phi(l, K)} - \beta^k\geq 0$.

        \item\label{case:(x)} If~$\phi(l, k) \in \brackets{i_k,j_k}$ and~$\phi(l, k) \neq \phi(l, k-1)$, then we can set~$\myalpha{=}{1}{k} = \myalpha{=}{2}{k} = 0$.
        Let~$l'$ be the other wire such that~$\brackets{\phi(l, k), \phi(l', k)} = \brackets{i_k, j_k}$.
        By the same argument, we can set~$\beta^k = \nicefrac{\left(\myw_{\phi(l', K)} + \myw_{\phi(l, K)}\right)}{2}$, and either~$\myalpha{\times}{1}{k} = \beta^k - \myw_{\phi(l, K)}$ and~$\myalpha{\times}{2}{k} = \myw_{\phi(l')} - \beta^k$ if~$\phi(l, k) < \phi(l', k)$ or~$\myalpha{\times}{1}{k} = \beta^k - \myw_{\phi(l', K)}$ and~$\myalpha{\times}{2}{k} = \myw_{\phi(l, K)} - \beta^k$ otherwise.
    \end{enumerate} 

    This construction means that for any~$l$ and~$k$, we have, depending on the constraint corresponding to~$x^k_{\phi(l, k)}$, either~$\delta^k_{\phi(l, k)} = \myw_{\phi(l, K)}$,~$\beta^k - \myalpha{=}{1}{k} - \myalpha{\times}{1}{k} = \myw_{\phi(l, K)}$ or~$\beta^k + \myalpha{\times}{2}{k} + \myalpha{=}{2}{k} = \myw_{\phi(l, K)}$.
    Plugging those values into the System~\eqref{eq:all-dual-constraints} immediately satisfies constraints~\eqref{eq:dual-NK},~\eqref{eq:dual-UK} and~\eqref{eq:dual-DK}.
    The remaining constraints reduce to only three different cases:
    \begin{alignat}{3}
        \myw_{\phi(l, K)} & -\delta^{k+1}_{\phi(l, k+1)} && \leq 0,
        \label{eq:propagate-N}\\
        \myw_{\phi(l, K)} & -\left(\beta^{k+1} - \myalpha{=}{1}{k+1} + \myalpha{\times}{2}{k+1}\right) && \leq 0,
        \label{eq:propagate-U}\\
        \myw_{\phi(l, K)} & -\left(\beta^{k+1} - \myalpha{\times}{1}{k+1} + \myalpha{=}{2}{k+1}\right) && \leq 0.
        \label{eq:propagate-D}
    \end{alignat}
    Equation~\eqref{eq:propagate-N} implies that at step~$k+1$,~$\phi(l, k)\notin\brackets{i_{k+1}, j_{k+1}}$ and thus~$\phi(l, k+1) = \phi(l, k)$.
    As such,~$\delta^{k+1}_{\phi(l, k+1)} = \myw_{\phi(l, K)}$ and the equation is satisfied.
    Equation~\eqref{eq:propagate-U} implies that~$\phi(l, k) = i_{k+1}$.
    If then~$\phi(l, k+1) = \phi(l, k)$, we are in Case~\ref{case:(=)} and, on the other hand, if~$\phi(l, k+1) \neq \phi(l, k)$, we are in Case~\ref{case:(x)}.
    Either case sets the correct~$\alpha$ to zero such that the inequality holds.
    Equation~\eqref{eq:propagate-D} works analogously.

    In summary, we have defined a feasible dual solution whose objective value is given by~$\sum_{l=1}^n \tilde{x}^K_l \myw_{\phi(l, K)}$,
    which serves as a certificate for optimality of the primal solution~$\left(\tilde{x}^0, \ldots, \tilde{x}^K\right)$.   
\end{proof}

We are now able to prove the main statement of this section, namely that there exists a compact extended formulation of separation polyhedra for LCIs of sparse knapsack polytopes.
As we have established that there are only polynomially many different equivalence classes of LCIs, it is sufficient to show that our extended formulation applies to a single class~$\mathcal{M}(C, S)$.

\begin{theorem}
    Let~$\hat{x}\in\cube{n}$ and let~$C$ be a minimal cover and~$S$ be a corresponding maximal independent set for a $\sparsity$-sparse knapsack.
    Let~$\bar{P}$ and~$P$ be as defined in~\eqref{eq:polytope-all-cuts} and~\eqref{eq:all-sorted-networks}.
    Then~$\hat{x}\in\bar{P}$ if and only if there exists a point in~$P$ satisfying Constraint \eqref{eq:strongest-cut} applied to~$x^K$. 
\end{theorem}
\begin{proof}
    On the one hand, if~$\hat{x}\notin\bar{P}$, there exists a pair $(C', S') \in \mathcal{M}(C,S)$ generating a violated lifted cover inequality.
    Then so does the strongest representative~$(C, S)^{\max}$.
    This means that Inequality~\eqref{eq:strongest-cut} does not hold for the sorted copy of~$\hat{x}$.
    At the same time, by replacing the coefficients~$\myw_i$ in Proposition~\ref{thm:sorting-network-dual-primal} with the~$\nu_j(i)$ for all~$i \in \naturalsto{\card{W_j}}$ and~$j \in \naturalsto{\sparsity}$, we have that
    \[
      \min\brackets{\sum_{i=1}^{\card{W_j}} \nu_j(i)\cdot x^K_{j_i} : x\in P} \geq \sum_{i=1}^{\card{W_j}} \nu_j(i) \cdot\hat{x}_{j_i}.
    \]
    Therefore their sum over all~$j\in\naturalsto{\sparsity}$ will exceed~$\card{C}-1$.
    As a consequence, the~$K$-th component of each point~$(x^0,\dots,x^K) \in P$ violates~\eqref{eq:strongest-cut}.

    On the other hand, if all LCIs within family~$\mathcal{M}(C,S)$ are satisfied, then Proposition~\ref{thm:sorting-network-dual-primal} gives a solution~$\tilde{x}^K$ such that~$\sum_{i=1}^{\card{W_j}} \nu_j(i)\cdot \tilde{x}^K_{j_i} = \sum_{i=1}^{\card{W_j}} \nu_j(i) \cdot\hat{x}_{j_i}$ and consequently
    \begin{equation*}
        \sum_{j=1}^\sparsity \sum_{i=1}^{\card{W_j}} \nu_j(i)\cdot \tilde{x}^K_{j_i} 
        =\sum_{j=1}^\sparsity \sum_{i=1}^{\card{W_j}} \nu_j(i) \cdot\hat{x}_{j_i} \leq \card{C} -1.
      \end{equation*}
      That is, \eqref{eq:strongest-cut} is satisfied by the~$K$-th component of some point in~$P$.
\end{proof}

\section{Practical Aspects of Using LCIs} 

\label{chap:practical-implementation}

\noindent
We have shown that we can list all non-equivalent minimal covers as well as listing their corresponding maximal independent sets in polynomial time if the knapsack is sparse.
In this section we give a brief overview of some practical considerations that we will make use of in an implementation of the ideas disclosed before.
As said in the proof of Theorem~\ref{thm:complexitySparse}, equivalence classes of the~$\sim$ relationship are uniquely defined by the amount of elements in each weight class selected.
Therefore we can represent sets with short~$\sparsity$-dimensional arrays whose entries correspond to the different weights.
Formally, any set~$S\subseteq\naturalsto{n} = W_1 \cup \dots \cup W_{\sparsity}$ will be written as a tuple~$\left(s_1, \dots, s_\sparsity\right)$ with~$s_j = |S\cap W_j|$ for all~$j\in\naturalsto{\sparsity}$.

\paragraph{Getting Minimal Covers}
The simplest way to find all non-equivalent covers is to exhaustively inspect all tuples, from~$(0, \ldots, 0)$ to~$(|W_1|, \ldots, |W_\sparsity|)$.
The tuple~$(c_1, \ldots, c_\sparsity)$ corresponds to a cover family if~$\sum_{j=1}^\sparsity w_j c_j > \beta$.
Note that since this search is exhaustive, it is no better than any brute-force algorithm.
We settled on a basic reverse lexicographical ordering.
That is, we start with the tuples~$(1, 0, \dots, 0), (2, 0, \dots, 0)$ until~$(\card{W_1}, 0, \dots, 0)$ before inspecting~$(0, 1, 0\dots ,0), (1, 1, 0\dots, 0)$ and so on.
This ordering allows for a couple of enhancements.
\begin{itemize}
    \item Reversing the enumeration. 
        When~$\sum_{i=1}^n a_i \leq 2\beta$, one arguably might need many items in a cover.
        It can then be faster to start from the largest cover and go down to minimal covers.

    \item Skipping steps when the current set is a minimal cover.
        When~$(c_1, \dots, c_{\sparsity})$ is a minimal cover, then all subsequent covers~$(c_1 + 1, c_2, \dots, c_{\sparsity})$ to~$(\card{W_1}, c_2, \dots, c_{\sparsity})$ cannot be minimal.
        We can then skip these~$\card{W_1} - c_1$ iterations.
    \item Make the increment step larger.
        In a similar way to skipping non-minimal covers, one can test if a non-covering set~$(c_1, \dots, c_{\sparsity})$ becomes a cover when replacing~$c_1$ by~$\card{W_1}$.
        If it does, then we can find the minimal one in between with the default enumeration.
        If it does not, then all the steps in between can be skipped. 
    \item Finding the first minimal cover in constant time.
        This is done by iteratively finding how many elements of the~$j$-th weight class are needed to complete the cover assuming the first~$j-1$ of them are all selected, for all~$j =\sparsity$ down to~$1$.
        Each iteration needs only one division with remainder so the total runtime is~$\Oh{\sparsity}$. 
\end{itemize}

\paragraph{Getting the Lifting Coefficients}
Recall that, to obtain a facet-defining inequality from a cover inequality, we need to compute the corresponding~$\mu$ function,~$\pi$ coefficients, and find a maximal independent set~$S$.
Given a minimal cover~$C$ in the form~$(c_1, \dots, c_{\sparsity})$, the values of~$\mu(h)$ and the~$\pi$ coefficients follow immediately.
The generation of maximal independent sets is not as straightforward.
While we could again list all possible non-equivalent sets~$S$, and test if Inequality~\eqref{eq:def_indep_set} holds, independence also requires that all proper subsets~$Q$ are independent.
A naive listing that keeps track of invalid subsets with smaller cardinality is potentially too memory-intensive.
We suggest with Algorithm~\ref{alg:find_next_indepset} a different approach that will considerably lighten the memory burden as well as speeding up the procedure as it does not inspect all possible~$\Oh{n^{\sparsity}}$ sets~$S$.
These benefits come at the expense of potentially skipping certain types of independent sets.
The motivation behind the algorithm comes from a~$2D$ visualization of the criterion in~\eqref{eq:def_indep_set} as we explain next.

The two quantities that change for each set~$S$ in Equation~\eqref{eq:def_indep_set} are~$\sum_{i\in S} a_i$ as well as~$\sum_{i\in S} (\pi_i + 1)$ which can be rewritten as~$a(S)$ and~$\pi(S) + \card{S}$, respectively.
In particular, Inequality~\eqref{eq:def_indep_set} can be seen as a constraint in two dimensions, namely~$y > \mu(x) - \Delta$ when replacing
\begin{equation*}
    \left(x_S, y_S\right),\text{ where } x_S = \pi(S) + \card{S} \text{ and } y_S=a(S).
\end{equation*}
In this representation, we can visualize the location of points~$(x_Q, y_Q)$ for the subsets~$Q\subsetneq S$ in a 2D plot.
In particular, it is in principle possible that two distinct sets~$S, S'$ could end on the same point~$(x_S, y_S)= (x_{S'}, y_{S',})$, but it cannot happen if~$S\subsetneq S'$.
\begin{observation}
    If~$S\subsetneq S'$ then~$(x_S, y_S)< (x_{S'}, y_{S',})$.
\end{observation}
\begin{proof}
    If~$S$ is a strict subset of~$S'$, then~$x_S < x_{S'}$ and~$y_S < y_{S'}$.
    This follows from the fact that~$a(S) < a(S')$. 
    Using the same reasoning for~$x_S$, $(x_S, y_S)< (x_{S'}, y_{S',})$ follows.
\end{proof}

\paragraph{Jumps and Slopes}
With this~$(x, y)$ representation, all singletons~$\brackets{i}$ from each weight class~$W_j$ have~$(x_{\brackets{i}}, y_{\brackets{i}}) = (\pi_j +1, w_j)$.
Since each set~$S$ consists of only~$\sparsity$ different weights types, adding an element of~$W_j$ to the set~$S$ is equivalent to moving the point to~$\left(x_S + \pi_j + 1, y_S + w_j\right)$. 
We refer to such a movement as a \emph{jump} of~$j$.
The point~$(x_S, y_S)$ can then be viewed as the end point of a sequence of jumps form~$(x_{Q_0}, y_{Q_0})$ to~$(x_{Q_{\card{S}}}, y_{Q_{\card{S}}})$, where we call the subsets~$\emptyset = Q_0 \subsetneq Q_1 \subsetneq \dots \subsetneq Q_{\card{S}} = S$ a \emph{jump sequence} of~$S$, as they differ by one element each.
Using this representation and~\eqref{eq:def_indep_set}, a set~$S$ is independent if and only if all jump sequences~$ Q_0 \subsetneq Q_1 \subsetneq \dots \subsetneq Q_{\card{S}}$ are above the boundary~$y = \mu(x)-\Delta$ (see Figure~\ref{fig:2d_representation_indepset}).
\begin{figure}[!htbp]
    \centering
    \begin{tikzpicture}[scale=0.3]
    \draw[step=2cm, gray!20, very thin] (0, -1) grid (15, 15);
    \draw[->, black] (0, -2) -- ( 0, 15) node[anchor=south east] {$y$};
    \draw[->, black] (0,  0) -- (15,  0) node[anchor=north west] {$x$};

    \draw[\myblue] ( 0, -2) -- ( 2,  2) -- (12, 12) -- (15, 12) node[anchor=south] {$\mu(x) - \Delta$};
    \draw[\myblue, dashed] (15, 12) -- (16, 12);

    \draw[black, -stealth]  (0, 0) -- (4, 8)    node[anchor=south east] {$\K$};
    \filldraw[black, radius=2pt] (4, 8) circle;
    \filldraw[black, radius=2pt] (2, 6) circle;
    \draw[\mygreen, -stealth]  (0, 0) -- (2, 6)    node[anchor=south east] {$\Q$};
    \filldraw[black, radius=2pt] (4, 12) circle node[anchor=south]      {$2\Q$};
    \filldraw[black, radius=2pt] (6, 14) circle node[anchor=south]      {$\Q + \K$};
    \draw[\myred, -stealth]    (0, 0) -- (4, 2)    node[anchor=west]       {$\J$};
    \draw[\mygreen, -stealth]  (2, 6) -- (6, 8);
    \draw[\myred, -stealth]    (4, 2) -- (6, 8)    node[anchor=south]      {$\Q + \J$};

\end{tikzpicture}
    \caption{
        $2D$ visualization of independent sets for a given knapsack.
        In blue the boundary delimits the region above which the inequality~\eqref{eq:def_indep_set} holds.
        In red and green different jump sequences leading to the set~$(1, 1, 0)$. The red sequence highlights the fact that~$(1, 0, 0)$, one element of weight~$w_1$, is not independent, and therefore~$(1, 1, 0)$ cannot be independent.
    }
    \label{fig:2d_representation_indepset}
\end{figure}
Note that~$\mu$ is only defined for positive integer values, but in the following figures we will extend it linearly between each integer points.
We allow ourselves this abuse of notation as all the points we will compare to the boundary~$y=\mu(x)-\Delta$ will have integer coordinates.
\begin{observation}
    The set~$\brackets{(x, y)\in \R_+\times\R_+ : y \leq \mu(x) - \Delta}$ is convex.
    \label{obs:mu-is-concave}
\end{observation}
\begin{proof}
    It suffices to show that the linear extension of~$\mu$ is concave.
    When~$h>\card{C}$,~$\mu(h) = \mu(\card{C})$ therefore we only need to show that~$\mu$ is concave between~$0$ and~$\card{C}$.
    Since~$\mu(h)$ is the sum of the~$h$ heaviest weights in~$C$, the slope between~$h$ and~$h+1$ of~$\mu$ is~$\mu(h+1)-\mu(h) = a_{j_{h+1}}$.
    Because~$a_{j_{i}} \leq a_{j_{h}}$ for any~$h<i$, we then have that the slopes are not increasing, meaning~$\mu$ must be concave.
\end{proof}
To check whether all jump sequences are above the boundary~$y = \mu(x) - \Delta$, the next lemma states that it is not necessary to inspect all~$\card{S}!$ orderings of~$S$.
Instead, it is sufficient to check one particular jump sequence.
\begin{lemma}
    \label{lemma:convex-path-vs-border}
    Let~$\mu\colon\R_+\rightarrow \R$ be any function.
    Let~$v_1 \leq \dots \leq v_n\in\R_+$ be scalars and~$\gamma$ a permutation on~$\naturalsto{n}$.
    We define~$f_{\gamma}\colon[0, n]\rightarrow\R$ the piecewise linear function with breakpoints~$\brackets{0, 1,\dots, n}$ and slopes~$(v_{\gamma(1)}, \dots, v_{\gamma(n)})$.
    If there exists a permutation~$\gamma'$ and a real~$s\in[0, n]$ such that~$f_{\gamma'}(s) \leq \mu(s)$, then~$f_{\text{Id}}(s) \leq \mu(s)$.
\end{lemma}
\begin{proof}
    Since~$v_1 \leq \dots\leq v_n$, we have that for any permutation~$\sum_{k=1}^h v_k \leq \sum_{k=1}^h v_{\gamma(k)}$.
    This means that~$f_{\gamma}(h) \geq f_{\text{Id}}(h)$ for all integer~$h$, and it easily extends to real values as well since~$f_{\gamma}$ is linear between integer values.
\end{proof}
This lemma motivates Algorithm~\ref{alg:greedy_indep_set} to find maximal independent sets:
We will build independent sets by making jumps such that the corresponding piecewise linear function is convex and stays in the region~$y > \mu(x)-\Delta$.
Drawing the graph of the function for all possible sets~$S$ will then have a tree-like structure and the maximal independent sets correspond to the endpoints of the branches that have not touched the boundary.
We devise a depth-first search algorithm to list all these endpoints.
Note that in the~$2D$ representation some branches may seem to connect or overlap, but the implicit structure is still that of a tree (see Figure~\ref{fig:tree-of-paths}).
We first reorder the weight classes~$W_j$ for all~$j\in\naturalsto{\sparsity}$ by comparing their \emph{slope}~$p_j$.
The algorithm can then explore this tree by choosing the smallest slope at each branching.
\begin{figure}[tb]
    \centering
    \begin{tikzpicture}[scale=0.25]
        \draw[step=2cm, gray!20, very thin] (0, 0) grid (17, 15);
        \draw[->, black] (0,  0) -- ( 0, 15) node[anchor=north east] {$y$};
        \draw[->, black] (0,  0) -- (17,  0) node[anchor=south] {$x$};

        \draw (0, 0) -- (4, 2) --(8, 4) -- (10, 6) -- (12, 8) -- (14, 10) -- (16, 14);
        \draw (12, 8) -- (14, 12);
        \draw (10, 6) --(12, 10);
        \draw (8, 4) -- (10, 8);
        \draw (4, 2) -- (6, 4) -- (8, 6) -- (10, 8) -- (12, 12);
        \draw (8, 6) -- (10, 10);
        \draw (6, 4) -- (8, 8);
        \draw (4, 2) -- (6, 6);
        \draw (0, 0) -- (2, 2) --(4, 4) --(6, 6) -- (8, 10);
        \draw (4, 4) -- (6, 8);
        \draw (2, 2) -- (4, 6);
        \draw (0, 0) -- (2, 4);

        \draw[very thick, ->] (0, 0) -- (4, 2) node[anchor=north west] {$a$};
        \draw[very thick, ->] (0, 0) -- (2, 2) node[anchor=west] {$b$};
        \draw[very thick, ->] (0, 0) -- (2, 4) node[anchor=south east] {$c$};
        
        \filldraw   ($(10, 8) + (10pt, 0pt)$) -- 
                    ($(10, 8) - (0pt, 10pt)$) --
                    ($(10, 8) - (10pt, 0pt)$) --
                    ($(10, 8) + (0pt, 10pt)$);
        \filldraw (6, 6) circle[radius=6pt];
        \filldraw[white] ($(10, 8) + (5pt, 0pt)$) -- 
                    ($(10, 8) - (0pt, 5pt)$) --
                    ($(10, 8) - (5pt, 0pt)$) --
                    ($(10, 8) + (0pt, 5pt)$);
        \filldraw[white] (6, 6) circle[radius=3pt];
    \end{tikzpicture}
    \caption{
        Union of all convex paths with at most~$2$ times~$a=(2, 1)$, ~$3$ times~$b=(1, 1)$ and once~$c=(1, 2)$. 
        Note that the white dot and white diamond appear to be on two paths at the same time.
        However, with the diamond for example, one arises from the set~$(2a, 0b, 1c)$ and the other from~$(1a, 3b, 0c)$ so neither is a subset of the other.
    }
    \label{fig:tree-of-paths}
\end{figure}

\begin{algorithm}[htb]
    \label{alg:greedy_indep_set}
    \SetAlgoLined
    \DontPrintSemicolon
    \SetKwInOut{Input}{input}
    \SetKwInOut{Output}{output}
    \Input{an array~$s$ representing an independent set, an index~$m$ and a permutation~$\gamma$ on~$\naturalsto{\sparsity}$ such that~$p_{\gamma(j)} \leq p_{\gamma(j+1)}$ for all~$j<\sparsity$.}
    \Output{$s$ representing an independent set, maximal with respect to the first~$m$ fixed entries}
    $(x, y) \gets (0, 0)$\; 
    \For{$j = 1$ to~$m$}{
        $(x, y) \gets (x, y) + s_{\gamma(j)} \cdot (\pi_{\gamma(j)} + 1, w_{\gamma(j)})$\tcp*{read fixed part}
    }
    \For{$j = m + 1$ to~$\sparsity$}{
        $s_{\gamma(j)} \gets 0$\tcp*{greedy alg. on the remaining entries}
        \For{$k = 1$ to~$|W_{\gamma(j)}| - |C\cap W_{\gamma(j)}|$}{
            \eIf{$y + w_{\gamma(j)} > \mu(x + (\pi_{\gamma(j)} + 1)) - \Delta$}{
                $(x, y) \gets (x, y) + (\pi_{\gamma(j)} + 1, w_{\gamma(j)})$\;
                $s_{\gamma(j)} \gets s_{\gamma(j)} + 1$\;
            }{ 
                break\;
            }
        }
    }
    \Return{$s$}\;
    \caption{Greedily find a maximal independent set while preserving the first~$m$ entries}
\end{algorithm}

The first independent set can be found via a greedy search, which is what Algorithm~\ref{alg:greedy_indep_set} does when~$m = 0$.
Start the branch with the jumps of the weight class with the smallest ratio~$p_j = \frac{w_j}{\pi_j + 1}$.
Note that it is possible for different weights to have the same slopes.
In that case prioritize the one whose jump is the longest, or equivalently whose coefficient~$\pi_j$ is the largest.
We can then iteratively take as many elements of the current weight class as possible,  until it is either empty or the branch reaches the boundary, before considering the next smallest slope to find a maximal independent set.
Such algorithm necessarily produces a sequence whose function in our~$2D$ representation will be convex.

Intuitively, the next independent set can be found by going back a few steps on the branch, and choosing a larger slope earlier (making the resulting function still convex, but slightly steeper).
Assuming the branch we were at ended with the set~$(s_1, \dots, s_{\sparsity})$, let~$j^\star = \argmax\brackets{j\in \naturalsto{\sparsity} : s_{\gamma(j)} > 0}$.
Going one step back on the branch would end on the point where~$s_{\gamma(j^\star)}\gets s_{\gamma(j^\star)} -1$.
One can then get a new convex function by starting from this point and using the same greedy search, but on the remaining entries~$j^\star + 1$ to $\sparsity$.
This is what Algorithm~\ref{alg:find_next_indepset} does inside the \texttt{while} loop.
The~$m$ parameter in Algorithm~\ref{alg:greedy_indep_set} indicates how many of the~$s_1, \dots, s_{\sparsity}$ to fix before the greedy search between lines~$5$ to $15$.
Observe that the output independent sets that are not maximal will appear right after the ones they are subset of.
This is a consequence of the algorithm being a depth-first search.
In particular, if the current independent set has~$s_{\gamma(\sparsity)} > 0$ then this algorithm will next output the same set but with~$s_{\gamma(\sparsity)}\gets s_{\gamma(\sparsity)} - 1$ as next independent set.
This is why we skip index~$\sparsity$ in Algorithm~\ref{alg:find_next_indepset} as we know that these will not be maximal anyways.
In general, we only need to keep track of the current independent set and compare it to the new one to check for maximality.
\begin{algorithm}[tb]
    \label{alg:find_next_indepset}
    \SetAlgoLined
    \DontPrintSemicolon
    \SetKwInOut{Input}{input}
    \SetKwInOut{Output}{output}
    \Input{an array~$s$ representing a maximal independent set.}
    \Output{$s$ representing a new maximal independent set if possible, otherwise outputs~$0$.}
    $s^{\text{init}} = s$\;
    \While{$s \leq s^{\text{init}}$}{
        $j^\star \gets \max\brackets{j : 1 \leq j \leq \sparsity - 1, s_{\gamma(j)} > 0}$\; 
        $s_{\gamma(j^\star)} \gets s_{\gamma(j^\star)} -1$\;
        run Algorithm~\ref{alg:greedy_indep_set} on~$s$ with~$m= j^\star$\;
    }
    \Return{$s$}\;
    \caption{Finds the next maximal independent set}  
\end{algorithm}

While Lemma~\ref{lemma:convex-path-vs-border} gives a guarantee to find a function that necessarily violates~$y>\mu(x) - \Delta$ if any of its reorderings also does, it only does so for continuous functions.
In the previous discussion, the branches are made of discrete jumps.
However, since~$\mu$ is concave, it is possible that one jump passes under the boundary and ends sufficiently far to land back in the feasible region.
This special case can be detected by splitting jumps of length~$\pi_j + 1$ in~$x$-direction into several jumps of length~$1$.
Unfortunately it does not necessarily mean that the set it corresponds to is not independent (see edge case in Figure~\ref{fig:complicated_indepset}).
For our current implementation, we have decided to only allow for branches that do not intersect the boundary in any way.
These cases will then result in some non-maximal independent sets, and hence the corresponding lifted cover inequalities are not facet-defining.
We settled on this tradeoff for computational time as these edge cases were very rarely observed during our tests.

\begin{figure}[htb]
    \centering
    \begin{tikzpicture}[scale=0.5]
    \draw[very thin, gray!20, step=1] (0, 0) grid (4.5, 5.5);
    \draw[black, ->] (0, 0) -- (4.5, 0) node[anchor=north west] {$x$};
    \draw[black, ->] (0, 0) -- (0, 5.5) node[anchor=south west] {$y$};

    \draw[\myblue] (1, 0) -- (2, 3) -- (4, 3);
    \draw[\myblue, dashed] (4, 3) -- (5, 3) node[anchor=north] {$\mu(x) - \Delta$};

    \draw[\mygreen, -stealth] (0, 0) -- (1, 1) node[anchor=west      ] {$1$};
    \draw[\myred, -stealth] (0, 0) -- (2, 3) node[anchor=north west] {$3$};
    \draw[\mygreen, -stealth] (0, 0) -- (2, 4) node[anchor=south east] {$4$};

    \draw (8, 2.5) node {$\Rightarrow$};
\end{tikzpicture}
\begin{tikzpicture}[scale=0.5]
    \draw[very thin, gray!20, step=1] (0, 0) grid (4.5, 5.5);
    \draw[black, ->] (0, 0) -- (4.5, 0) node[anchor=north west] {$x$};
    \draw[black, ->] (0, 0) -- (0, 5.5) node[anchor=south west] {$y$};

    \draw[\myblue] (1, 0) -- (2, 3) -- (4, 3);
    \draw[\myblue, dashed] (4, 3) -- (5, 3);

    \draw[\mygreen, -stealth] (0, 0) -- (1, 1);
    \draw[\myred, -stealth] (0, 0) -- (2, 3);

    \draw[dashed, \myred] (1, 1) -- (3, 4) node[anchor=south] {$1+3$};
    \draw[\myred, -stealth] (2, 3) -- (3, 4);

    \filldraw[white] (5, 0) rectangle (6, 1);
\end{tikzpicture}
\begin{tikzpicture}[scale=0.5]
    \draw[very thin, gray!20, step=1] (0, 0) grid (4.5, 5.5);
    \draw[black, ->] (0, 0) -- (4.5, 0) node[anchor=north west] {$x$};
    \draw[black, ->] (0, 0) -- (0, 5.5) node[anchor=south west] {$y$};

    \draw[\myblue] (1, 0) -- (2, 3) -- (4, 3);
    \draw[\myblue, dashed] (4, 3) -- (5, 3);

    \draw[\mygreen, -stealth] (0, 0) -- (1, 1);
    \draw[\mygreen, -stealth] (0, 0) -- (2, 4);

    \draw[dashed, \myred, -stealth] (1, 1) -- (3, 5);
    \draw[\mygreen, -stealth] (2, 4) -- (3, 5) node[anchor=south] {$1+4$};
\end{tikzpicture}
    \caption{
        Example of a knapsack cover whose independent set is difficult to compute. 
        Let the weights be~$\brackets{1, 3, 4}$ and capacity~$3$. 
        For the cover~$C = (0, 2, 0)$ the only maximal independent set is the union of one element of weight~$1$ and all the available elements of weight~$4$.
        If the algorithm does not check for jumps passing under the boundary, it would wrongly declare the set~$(1, 1, 0)$ as independent.
        If it does check for intersection with the boundary, it would not find the independent set~$(1, 0, 1)$.
    }
    \label{fig:complicated_indepset}
\end{figure}

\paragraph{Incorporating GUBs}
Another way to strengthen the LCIs even further is to make use of other information from the MIP instance the knapsack arose from.
One useful type of constraint is a group of $x(L_i) \leq 1$ for some non-overlapping sets~$L_1, \dots, L_m$.
When~$\card{L_i} = 1$ the inequality reduces to a classical variable bound.
We can then assume without loss of generality that these constraint partition the variable-space.
These are commonly referred to as \emph{generalized upper bound} constraints, or GUBs in short~\cite{dantzig1967generalized}. 
We can combine these with our lifted cover inequalities to strengthen the cuts.
Recall that one special case for our LCIs was when a weight class~$W_j$ induced coefficients~$\pi_j = 0$. 
Then all coefficients in the LCI for that weight class are either zero or one.
We can then augment the inequality by setting the coefficients with indices~$i\in W_j\setminus (C\cup S)$ that share a GUB with some other~$i'\in W_j\cap(C\cup S)$ to one.
In other words, we incorporate some information from the GUB into the LCI, which do not necessarily align with the sparsity patterns of the knapsack as they are ``external''.
This addition guarantees a fair comparison with benchmark methods as we have observed that they frequently leverage GUB information.

\section{Numerical Experience}
\label{sec:numerics}

\noindent
In the preceding sections, we have discussed two approaches for exploiting
lifted cover inequalities (LCIs) when solving mixed-integer programs
containing sparse knapsack constraints:
an extended formulation, which adds a polynomial number of auxiliary
variables and constraints to enforce that a solution adheres to all LCIs,
as well as a separation algorithm that separates LCIs for sparse knapsack
constraints in polynomial time.
This section's aim is to investigate the impact of these two approaches on
solving mixed-integer programs.
In Section~\ref{sec:numericsEF}, we focus on extended formulations for LCIs
for a particular class of knapsack polytopes, whereas
Section~\ref{sec:numericsSepa} reports on numerical experience of
separating LCIs for sparse knapsacks without using auxiliary variables.

\paragraph{Computational Setup}
All our techniques have been implemented in the open-source solver
\scip~9.0.1~\cite{bolusani2024scip} with LP-solver \solver{Soplex}~7.0.1.
\scip has been compiled with the external software
\solver{sassy}~1.1~\cite{AndersEtAl2023} and
\solver{bliss}~0.77~\cite{JunttilaKaski2011} for detecting
symmetries.
Our implementation is publicly available at GitHub\footnote{\url{https://github.com/Cedric-Roy/supplement_sparse_knapsack}} and~\cite{release}.

All of the following experiments have been conducted on a Linux cluster
with Intel Xeon E5-1620 v4~\SI{3.5}{\giga\hertz} quad core processors
and~\SI{32}{\giga\byte} of memory.
The code was executed using a single thread.
When reporting the mean of~$n$ numbers~$t_1,\dots,t_n$, we use the shifted
geometric mean~$\prod_{i=1}^n (t_i + s)^{\nicefrac{1}{n}} - s$ with
shift~$s=1$ to reduce the impact of outliers.

The implementation follows the principles explained in Section~\ref{chap:practical-implementation}.
Namely for each knapsack inequality, we exhaustively iterate over all non-equivalent minimal covers, and for each cover we use our modified search (Algorithm~\ref{alg:find_next_indepset}) of independent sets to create non-equivalent lifted cover inequalities.
To separate LCIs, we use the separation algorithm described in the proof of Theorem~\ref{thm:complexitySparse}, i.e., for a sorted point~$\bar{x}$, we find for every family of minimal covers~$C$ and independent sets~$S$ an LCI with maximum left-hand side value w.r.t.~$\bar{x}$ in~$\Oh{n}$ time.
This LCI is possibly enhanced by GUB information as described in the previous section, and used as a cutting plane if it is violated by~$\bar{x}$.
The implementation of the extended formulation via sorting network underwent a preliminary test setup described in the following section.

\subsection{Evaluation of the Extended Formulation}
\label{sec:numericsEF}

\noindent
Our first experiment concerns the impact of extended formulations for LCIs.
In contrast to the results of Section~\ref{sec:sorting-network-polytope} that show how sorting
networks can be used to derive an extended formulation for LCIs for
arbitrary (sparse) knapsacks, we focus on a particular class of knapsacks,
so-called orbisacks~\cite{KaibelLoos2011}, which we will explain in more
details below.
The motivation for considering orbisacks rather than general knapsacks is
two-fold.
On the one hand, orbisacks arise naturally in many problems.
This allows to draw conclusions on a broad range of instances and the
effect of handling LCIs via an extended formulation is less likely to be
biased by problem structures present in a narrow class of instances.
On the other hand, orbisacks have~$2^{\Theta(n)}$ many LCIs that can be
modeled via an extended formulation containing~$O(n)$ variables and
constraints.
In contrast to the general sorting networks of Section~\ref{sec:sorting-network-polytope}, we thus can
make use of a tailored implementation for orbisacks, which is arguably more
effective than using a general extended formulation that does not exploit
specific structures of the underlying knapsack.
The numerical results therefore can better reveal the potential of extended
formulations for handling LCIs.

\paragraph{Background on Orbisacks}
The orbisack~\cite{KaibelLoos2011} is defined as
\[
  \orbisack_n
  \define
  \conv
  \Big\{
  x \in \B{n \times 2} :
  \sum_{i = 1}^n 2^{n-i}(x_{i,2} - x_{i,1}) \leq 0
  \Big\},
\]
and the vertices of~$\orbisack_n$ are all binary matrices whose first
columns are not lexicographically smaller than their second columns.
Orbisacks can be used to handle symmetries in mixed-integer
programs~\cite{HojnyPfetsch2019} and many of the instances of the mixed-integer programming
benchmark library MIPLIB2017~\cite{MIPLIB} allow their symmetries to be handled
by orbisacks; cf.~\cite{PfetschRehn2019}.

Note that orbisacks are not standard knapsack polytopes, because the defining
inequality has positive and negative coefficients.
By replacing, for each~$i \in [n]$, variable~$x_{i,1}$ by~$\bar{x}_{i,1} = 1 -
x_{i,1}$, however, it can be turned into a standard knapsack polytope
\[
  \orbisacktrans_n
  =
  \conv
  \Big\{
  x \in \B{n \times 2} :
  \sum_{i=1}^n 2^{n-i}(x_{i,1} + x_{i,2}) \leq 2^n - 1
  \Big\},
\]
and all LCIs derived from~$\orbisacktrans_n$ can be transformed back into
facet defining inequalities for~$\orbisack_n$.
Since the vertices of~$\orbisacktrans_n$ are matrices, a minimal cover
consists of tuples~$(i,j)$ with~$i \in [n]$ and~$j \in \{1,2\}$.
The minimal covers~$C$ of~$\orbisacktrans_n$ are characterized by an
index~$i^\star \in [n]$ and a vector~$\tau \in \{1,2\}^{i^\star - 1}$ such
that~$C = \{(i^\star,1),(i^\star,2)\} \cup \{(i,\tau_i) : i \in [i^\star -
1]\}$; see~\cite[Prop.~4]{Hojny2020} applied to the consecutive partition
in which all cells have size~2.
Moreover, one can show that all sequential liftings of a minimal
cover~$C$ with~$i^\star > 1$ result in the LCI
\[
  x_{1,1} + x_{1,2} + \sum_{i = 2}^{i^\star - 1} x_{i,\tau_i} +
  x_{i^\star,1} + x_{i^\star,2} \leq i^\star;
\]
for~$i^\star = 1$, the unique LCI is~$x_{1,1} + x_{1,2} \leq 1$.
As a consequence, there are~$2^{n-1}$ LCIs.
In the original variable space of the orbisack, the latter inequality reads
as~$-x_{1,1} + x_{1,2} \leq 0$, whereas the former inequality turns into
\begin{equation}
  \label{eq:orbisackCover}
  -x_{1,1} + x_{1,2} -x_{i^\star,1} + x_{i^\star,2}
  -\sum_{\substack{i \in [2,i^\star -1]\colon\\ \tau_i = 1}} x_{i,1}
  -\sum_{\substack{i \in [2,i^\star - 1]\colon\\ \tau_i = 2}} x_{i,2}
   \leq i^\star - T(\tau) - 2,
 \end{equation}
 where~$T(\tau) = \card{\{i \in \{2,\dots,i^\star-1\} : \tau_i = 1\}}$.

\paragraph{Extended Formulations for Orbisacks}
We now turn to an extended formulation based on~$P$, the sorting network polytope from Section~\ref{sec:sorting-network-polytope}.
Let~$x \in [0,1]^{n \times 2}$ be the variable matrix associated with an
orbisack.
Moreover, we introduce variables~$y_i$ for~$i \in [2,n-1]$ together with
the inequalities
\begin{subequations}
  \label{eq:EForbisack}
  \begin{align}
    -x_{i,1} &\leq y_i, && i \in [2,n-1],\\
    x_{i,2} &\leq 1 + y_i, && i \in [2,n-1],\\
    -x_{1,1} + x_{1,2} &\leq 0, &&\\
    -x_{1,1} + x_{1,2} -x_{i^\star,1} + x_{i^\star,2} + \sum_{i=2}^{i^\star}
    y_i&\leq 0, && i^\star \in [2,n]\label{eq:EForbisackLCI}\\
    x_{i,j}& \in [0,1], && (i,j) \in [n] \times [2],\\
    y_i & \in [-1,0], && i \in [2,n-1].
  \end{align}
\end{subequations}
We claim that~\eqref{eq:EForbisack} defines an extended formulation of Section~\ref{sec:sorting-network-polytope}.
Indeed, due to the first two families of inequalities, $y_i \geq
\max\{x_{i,2}-1, -x_{i,1}\}$.
Define, for~$i^\star \in [2,n]$, vector~$\tau \in \{1,2\}^{[2,i^\star-1]}$
to take value~1 if and only if~$-x_i \geq x_{i,2} - 1$.
Then,
\begin{align*}
  \sum_{i=2}^{i^\star-1} y_i
  &\geq
  \sum_{\substack{i \in [2,i^\star -1]\colon\\ \tau_i = 2}} x_{i,2}
  -
  \sum_{\substack{i \in [2,i^\star -1]\colon\\ \tau_i = 1}} x_{i,1}
  -
  \card{\{i \in [2,i^\star - 1] : \tau_i = 2\}}\\
  &=
  \sum_{\substack{i \in [2,i^\star -1]\colon\\ \tau_i = 2}} x_{i,2}
  -
  \sum_{\substack{i \in [2,i^\star -1]\colon\\ \tau_i = 1}} x_{i,1}
  -
  (i^\star - T(\tau) - 2).
\end{align*}
Consequently, every vector~$x \in [0,1]^{n \times 2}$ for which there
exists~$y \in \R^{[n-2]}$ such that~$(x,y)$
satisfies~\eqref{eq:EForbisack}, Inequality~\eqref{eq:EForbisackLCI} implies
that~$x$ satisfies~\eqref{eq:orbisackCover}.
Moreover, if~$x$ violates an LCI~\eqref{eq:orbisackCover}, also no~$y$
exists such that~$(x,y)$ satisfies~\eqref{eq:EForbisack}.
Since~\eqref{eq:EForbisack} also contains the only LCI that is not of
type~\eqref{eq:EForbisackLCI}, namely~$-x_{i,1} + x_{i,2} \leq 0$,
System~\eqref{eq:EForbisack} is an extended formulation.

\paragraph{Implementation Details}
\scip offers many possibilities for handling symmetries of mixed-integer
programs.
The high level steps of symmetry handling within \scip are to compute
symmetries, check whether some symmetries form a special group that can be
handled by effective techniques, and use some basic techniques for the
remaining symmetries.
The propagation of orbisacks and separation of LCIs falls into the latter
category.
To enforce that orbisacks are used whenever possible in this category, we
set the parameter \texttt{misc/usesymmetry} to value~1 and
\texttt{propagating/symmetry/usedynamicprop} to \texttt{FALSE}.
Moreover, we introduced two new parameters.
The first parameter allows to switch between
\scip's default techniques for handling orbisacks and extended
formulations.
That is, we either use \scip's default techniques or an extended
formulation.
The second parameter controls the maximum value of~$n$, i.e., the number of
rows, that we allow in matrices constrained by orbisacks.
When the number of rows of an orbisack exceeds the value~$k$ of the
parameter, we still define an extended formulation for the orbisack, but we
restrict the LCIs to the first~$k$ rows.
Note that this still allows to solve an instance correctly, because
orbisacks are only used to handle symmetries, but are no model constraints.
The motivation for this parameter is to avoid a blow-up of the model, which
turns out to be useful as we will see below.

\paragraph{Numerical Results}
The aim of our experiments is to compare the approach of handling LCIs via
an extended formulation and an exact separation routine for LCIs.
To this end, we compare our extended formulation~\eqref{eq:EForbisack} with
the build-in propagation and separation routines for orbisacks.
Moreover, we compare our extended formulation for LCI separation with two
extended formulations~\cite{KaibelLoos2011} of the orbisack itself, i.e.,
their projection onto the original variables yields~$O_n$.
For our purposes, it is only important that the second extended formulation
has~$3n$ variables and~$6n$ constraints ($8n$ when including non-negativity constraints), whereas the third extended
formulation has~$4n$ variables and~$3n$ constraints ($7n$ when including non-negativity constraints).
For further details, we refer the reader to~\cite{KaibelLoos2011}.

We have conducted experiments on a test set of~191 instances of MIPLIB2017
for which \scip applies orbisacks in the setting mentioned above.
In the experiments, we test four different settings:
\begin{description}
\item[default] uses \scip's default techniques to handle orbisacks by
  propagation and separation; cf.~\cite{HojnyPfetsch2019};
\item[EF1] uses extended formulation~\eqref{eq:EForbisack};
\item[EF2] uses the extended formulation from~\cite{KaibelLoos2011} with
  fewer variables;
\item[EF3] uses the extended formulation from~\cite{KaibelLoos2011} with
  more variables.
\end{description}
Moreover, for the extended formulations, we limit the number of rows of
orbisacks to~10 and~30, respectively.
We use a time limit of~\SI{7200}{\s} per instance; instances not solved to
optimality contribute~\SI{7200}{\s} to the mean running time.
Moreover, we experienced numerical instabilities of the LP solver for some
instances, which led to an early termination of \scip; these instances have
been excluded from the evaluation to obtain unbiased results. 
To evaluate the impact of the different techniques based on the difficulty
of instances, we extracted different subsets of instances.
The subset denoted by~$(t,7200)$ refers to all instances that are solved by
at least one setting and one setting needed at least~$t$ seconds to solve
the instance.
In particular, the subset~$(0,7200)$ contains all instances that are solved
by at least one setting.

Table~\ref{tab:orbisack} summarizes the results of our experiments.
The columns of the table have the following meaning.
Column ``subset'' refers to the subset of instances as explained above;
``\#'' specifies the number of instances in the subset; column ``time''
provides the mean running time of the setting; ``solved'' reports on the
number of instances solved by a setting

\begin{table}[t]
  \caption{Comparison of extended formulations for orbisacks and separation of LCIs.}
  \label{tab:orbisack}
  \scriptsize
  \centering
  \begin{tabular*}{\textwidth}{@{}l@{\;\;\extracolsep{\fill}}rrrrrrrrr@{}}
    \toprule
    & & & & \multicolumn{6}{c}{max.\ 10 rows}\\
    \cmidrule{5-10}
    & & \multicolumn{2}{c}{default} & \multicolumn{2}{c}{EF1} & \multicolumn{2}{c}{EF2} & \multicolumn{2}{c}{EF3}\\
    \cmidrule{3-4}\cmidrule{5-6}\cmidrule{7-8}\cmidrule{9-10}
    subset & \# & time & solved & time & solved & time & solved & time & solved\\
    \midrule
(0,7200)    & (83) & \num{ 188.29} & \num{  79} & \num{ 213.82} & \num{  75} & \num{ 264.00} & \num{  71} & \num{ 213.40} & \num{  75}\\
(100,7200)  & (58) & \num{ 768.35} & \num{  54} & \num{ 918.88} & \num{  50} & \num{1195.15} & \num{  46} & \num{ 885.65} & \num{  50}\\
(1000,7200) & (41) & \num{1415.54} & \num{  37} & \num{1753.44} & \num{  33} & \num{2641.72} & \num{  29} & \num{1552.07} & \num{  33}\\
(3000,7200) & (28) & \num{1898.00} & \num{  24} & \num{2093.35} & \num{  20} & \num{3979.08} & \num{  16} & \num{1822.49} & \num{  20}\\
    \midrule
    & & & & \multicolumn{6}{c}{max.\ 30 rows}\\
    \cmidrule{5-10}
(0,7200)    & (85) & \num{ 205.21} & \num{  79} & \num{ 269.01} & \num{  72} & \num{ 268.94} & \num{  73} & \num{ 285.02} & \num{  74}\\
(100,7200)  & (60) & \num{ 866.74} & \num{  54} & \num{1172.31} & \num{  47} & \num{1209.97} & \num{  48} & \num{1222.90} & \num{  49}\\
(1000,7200) & (45) & \num{1420.78} & \num{  39} & \num{2050.87} & \num{  32} & \num{2310.72} & \num{  33} & \num{2114.53} & \num{  34}\\
(3000,7200) & (35) & \num{1839.53} & \num{  29} & \num{2498.05} & \num{  22} & \num{3027.23} & \num{  23} & \num{2747.60} & \num{  24}\\
    \bottomrule
  \end{tabular*}
\end{table}

Observe that our extended formulation performs on average better than EF2.
The EF3 extended formulation, in contrast, has a very similar running time to EF1.
If only ten rows are enabled, our extended formulation tends to be slightly slower than EF3, but when we allow~$30$ rows the trend is inverted.
Note that the running times of the default setting change between using~10 and~30 rows, because the corresponding set of instances changes slightly.
However, none of the extended formulations, with either setting, have a better running time than the default \scip settings.
A possible explanation is that the extended formulations increase the problem size, and thus it takes longer to solve LP relaxations. 
To confirm this conjecture our experiments revealed that, with the extended formulations EF1, EF2, and EF3, the solver has to spend between~$4.4$,~$8.9$ and~$23.7\%$ more iterations, respectively, solving the LP relaxation at the root node.
Recall that EF1 is as basic as a sorting network can be, with only~$n$ comparisons, with no wires in common (System~\eqref{eq:EForbisack} shows here~$3n$ variable and~$5n$ constraints). 
In contrast, the polytope~$P$ from Section~\ref{sec:sorting-network-polytope} is much larger with~$\Oh{n\log{n}^2}$ variables and~$\Oh{n\log{n}}$ constraints.
The results for EF2 indicate that formulations that require more constraints might hinder the solving speed, as EF3 indicates that using more variables does not help either.
We conclude that the additional strength of LCIs via extended formulations is small in comparison to the more challenging LP relaxation and therefore refrained from implementing the extended formulation based on sorting networks for general knapsacks.

\subsection{Evaluation of the Separation Algorithm}
\label{sec:numericsSepa}

\noindent
In a second experiment, we evaluate whether an exact separation routine for
LCIs of sparse knapsacks reduces the running time of \scip when
solving general MIP problems.
To this end, we have run \scip on all instances of MIPLIB2017 with a time
limit of~\SI{1}{\hour} and extracted all instances for which \scip generates a
knapsack constraint with sparsity~3 or~4.
This results in a test set of~183 instances.
Note that this test set also contains instances in which no sparse
knapsacks are present in the original formulation, because \scip can turn
globally valid cutting planes into knapsack constraints.
As before, we remove instances from the test set that result in numerical
instabilities for the LP solver.
To assess the effect of separating LCIs for sparse knapsacks, we compare
our separation algorithm for LCIs with \scip's internal separation
algorithms using various settings. 

We encode settings via a string \texttt{m-M-ABC}, where the letters have
the following meaning.
A knapsack is classified as sparse if its sparsity~$\sigma$
satisfies~$\text{\texttt{m}} \leq \sigma \leq \text{\texttt{M}}$.
The letters~\texttt{A}, \texttt{B}, and~\texttt{C} describe the behavior of the separation routines
for LCIs for sparse knapsacks, for \scip's default cutting planes applied
for sparse knapsacks, and for \scip's default cutting planes applied for
non-sparse knapsacks, respectively.
The letters~\texttt{A}, \texttt{B}, and~\texttt{C} take values~0, R, or~S, where~0 means that the
corresponding cut is not separated, R means the cuts are separated only at
the root node, and~S means that cuts are separated at every fifth layer of
the branch-and-bound tree.
For example, setting 3-4-0RS means that a knapsack is considered sparse if
its sparsity is between~3 and~4, the exact separation of LCIs for sparse
knapsacks is disabled, \scip's default cutting planes for sparse knapsacks
are only separated at the root node, and \scip's default cutting planes for
non-sparse knapsacks are separated at the root node and within the tree.
\scip's default settings are resembled by \texttt{3-4-0RR}.

\begin{table}[t]
  \caption{Comparison of separation algorithms for LCIs with a time limit of~2 hours and sparsity~4.}
  \label{tab:knapsack2}
  \scriptsize
  \centering
  \begin{tabular*}{\textwidth}{@{}l@{\;\;\extracolsep{\fill}}rrrrrrrrrrr@{}}
    \toprule
    & & \multicolumn{2}{c}{4-4-0RR} & \multicolumn{2}{c}{4-4-SSS} & \multicolumn{2}{c}{4-4-S0R} & \multicolumn{2}{c}{4-4-0SR} & \multicolumn{2}{c}{4-4-SSR}\\
    \cmidrule{3-4}\cmidrule{5-6}\cmidrule{7-8}\cmidrule{9-10}\cmidrule{11-12}
    subset & \# & time & solved & time & solved & time & solved & time & solved & time & solved\\
    \midrule
(0,7200)    & (88) & \num{ 135.10} & \num{  84} & \num{ 131.60} & \num{  84} & \num{ 140.94} & \num{  86} & \num{ 129.78} & \num{  86} & \num{ 130.07} & \num{  86}\\
(100,7200)  & (54) & \num{ 689.51} & \num{  50} & \num{ 663.58} & \num{  50} & \num{ 735.82} & \num{  52} & \num{ 649.41} & \num{  52} & \num{ 645.49} & \num{  52}\\
(1000,7200) & (26) & \num{2095.39} & \num{  22} & \num{2015.03} & \num{  22} & \num{2349.22} & \num{  24} & \num{1942.71} & \num{  24} & \num{1902.45} & \num{  24}\\
(3000,7200) & (12) & \num{2885.25} & \num{   8} & \num{2777.98} & \num{   8} & \num{3854.45} & \num{  10} & \num{2577.57} & \num{  10} & \num{2463.05} & \num{  10}\\
(6000,7200) & (8) & \num{2680.59} & \num{   4} & \num{2454.50} & \num{   4} & \num{3930.99} & \num{   6} & \num{2161.04} & \num{   6} & \num{1995.31} & \num{   6}\\
    \bottomrule
  \end{tabular*}
\end{table}

\begin{table}[t]
  \caption{Comparison of separation algorithms for LCIs with a time limit of~4 hours and sparsity~4.}
  \label{tab:knapsack4}
  \scriptsize
  \centering
  \begin{tabular*}{\textwidth}{@{}l@{\;\;\extracolsep{\fill}}rrrrrrrrrrr@{}}
    \toprule
    & & \multicolumn{2}{c}{4-4-0RR} & \multicolumn{2}{c}{4-4-SSS} & \multicolumn{2}{c}{4-4-S0R} & \multicolumn{2}{c}{4-4-0SR} & \multicolumn{2}{c}{4-4-SSR}\\
    \cmidrule{3-4}\cmidrule{5-6}\cmidrule{7-8}\cmidrule{9-10}\cmidrule{11-12}
    subset & \# & time & solved & time & solved & time & solved & time & solved & time & solved\\
    \midrule
(0,14400)    & (97) & \num{ 211.38} & \num{  91} & \num{ 205.90} & \num{  91} & \num{ 215.10} & \num{  92} & \num{ 200.17} & \num{  92} & \num{ 200.54} & \num{  91}\\
(100,14400)  & (63) & \num{1083.88} & \num{  57} & \num{1046.26} & \num{  57} & \num{1112.36} & \num{  58} & \num{1003.73} & \num{  58} & \num{ 997.93} & \num{  57}\\
(1000,14400) & (35) & \num{3544.45} & \num{  29} & \num{3426.39} & \num{  29} & \num{3667.75} & \num{  30} & \num{3214.85} & \num{  30} & \num{3158.97} & \num{  29}\\
(3000,14400) & (21) & \num{6020.06} & \num{  15} & \num{5874.46} & \num{  15} & \num{6571.03} & \num{  16} & \num{5302.42} & \num{  16} & \num{5118.09} & \num{  15}\\
(6000,14400) & (16) & \num{6975.85} & \num{  10} & \num{6842.58} & \num{  10} & \num{8100.89} & \num{  11} & \num{6140.62} & \num{  11} & \num{5829.78} & \num{  10}\\
    \bottomrule
  \end{tabular*}
\end{table}

\paragraph{Sparsity~4}
In a first experiment, we focused on knapsacks of sparsity~4 with a time
limit of \SI{2}{\hour}.
Our experiments are summarized in Table~\ref{tab:knapsack2}; the meaning of
columns is analogous to Table~\ref{tab:orbisack}.
The reason for not including a smaller sparsity in this first experiment is
that, when inspecting \scip's source code, it seems that \scip's greedy
heuristics are capable to detect most minimal covers.
Therefore, we expected most benefits for knapsacks with a higher sparsity.

As we can see from Table~\ref{tab:knapsack2}, \scip benefits from a more
aggressive separation of cutting planes for knapsacks, because the running
time of the default setting \texttt{4-4-0RR} improves when using \texttt{4-4-SSS}
by~\SI{2.6}{\percent} on all solvable instances and up
to~\SI{8.4}{\percent} on the hardest instances in subset~$(6000,7200)$.
To better understand the impact of separation routines for sparse
knapsacks, we disabled separation of non-sparse knapsacks within the tree
and either separate \scip's default cutting planes or LCIs using our
implementation via the settings \texttt{4-4-0SR} or \texttt{4-4-S0R},
respectively.
We observe that separating \scip's default cutting planes improves on the
setting \texttt{4-4-SSS}, whereas only separating our LCIs degrades the
performance substantially.
The results indicate that, although LCIs are facet defining for sparse
knapsack polytopes, our separation routine can yield weaker cutting planes
than \scip's default heuristic separation routine.
A possible explanation for this behavior is that \scip's built-in separation
routines exploit GUB information in a more effective way, thus better
linking knapsack constraints with further problem information.

When enabling both \scip's separation routines and our LCIs in setting
\texttt{4-4-SSR}, however, the performance of \texttt{4-4-0SR} remains
approximately unchanged for all solvable instances and improves with the
instances becoming more difficult.
For example, for subset~$(1000,7200)$, the performance improves
by~\SI{2.1}{\percent} and for the most difficult instances in
subset~$(6000,7200)$ an improvement of~\SI{7.7}{\percent} can be observed.
The separation of LCIs thus seems to be more effective for difficult
instances.

To confirm this conjecture, we have conducted analogous experiments with a
time limit of~\SI{4}{\hour} per instance, which are summarized in
Table~\ref{tab:knapsack4}.
This table has a similar pattern as Table~\ref{tab:knapsack2}, and indeed,
for the most challenging instances the performance of \texttt{4-4-0SR} can
be improved by also separating LCIs by~\SI{5.1}{\percent}.
We therefore conclude that separating facet-defining LCIs is most helpful
for difficult instances, where it can lead to great performance
improvements.
Easier instances, however, can effectively be solved by heuristically
separating lifted cover inequalities that incorporate GUB information.

\begin{table}[t]
  \caption{Comparison of separation algorithms for LCIs with a time limit of~2 hours and sparsity~3 or~4.}
  \label{tab:knapsack2a}
  \scriptsize
  \centering
  \begin{tabular*}{\textwidth}{@{}l@{\;\;\extracolsep{\fill}}rrrrrrrrrrr@{}}
    \toprule
    & & \multicolumn{2}{c}{3-4-0RR} & \multicolumn{2}{c}{3-4-SSS} & \multicolumn{2}{c}{3-4-S0R} & \multicolumn{2}{c}{3-4-0SR} & \multicolumn{2}{c}{3-4-SSR}\\
    \cmidrule{3-4}\cmidrule{5-6}\cmidrule{7-8}\cmidrule{9-10}\cmidrule{11-12}
    subset & \# & time & solved & time & solved & time & solved & time & solved & time & solved\\
    \midrule
(0,7200)    & (88) & \num{ 135.10} & \num{  84} & \num{ 140.83} & \num{  83} & \num{ 144.10} & \num{  85} & \num{ 129.60} & \num{  86} & \num{ 133.26} & \num{  85}\\
(100,7200)  & (54) & \num{ 689.51} & \num{  50} & \num{ 734.05} & \num{  49} & \num{ 762.07} & \num{  51} & \num{ 647.15} & \num{  52} & \num{ 675.92} & \num{  51}\\
(1000,7200) & (28) & \num{1885.76} & \num{  24} & \num{2136.79} & \num{  23} & \num{2320.46} & \num{  25} & \num{1749.03} & \num{  26} & \num{1951.04} & \num{  25}\\
(3000,7200) & (14) & \num{2380.22} & \num{  10} & \num{3056.59} & \num{   9} & \num{3768.98} & \num{  11} & \num{2138.39} & \num{  12} & \num{2679.82} & \num{  11}\\
(6000,7200) & (9) & \num{2101.11} & \num{   5} & \num{2766.90} & \num{   4} & \num{4174.44} & \num{   6} & \num{1699.63} & \num{   7} & \num{2428.77} & \num{   6}\\
    \bottomrule
  \end{tabular*}
\end{table}

\paragraph{Sparsity~3 and~4}
In a second experiment, we also considered knapsacks with sparsity~3.
Table~\ref{tab:knapsack2a} shows the summarized results.
In contrast to exclusively using our separation routine of LCIs for
knapsacks of sparsity~4, separating LCIs does not improve the performance
of \texttt{3-4-0SR}.
A possible explanation for this behavior is that, as mentioned above,
\scip's built-in heuristics for separating lifted cover inequalities are
good for knapsacks of sparsity~3.
For finding a violated LCI, it is thus not necessary to enumerate all
(families of) minimal covers and their possible liftings.
Although the time for finding all LCIs for sparse knapsacks is usually
small, it is still a disadvantage as it imposes, in particular for the easy
instances, some avoidable overhead.
Moreover, \scip's strategies for incorporating GUB information into cover
inequalities could be stronger than our strategy.

Another explanation is that non-fully lifted cover inequalities tend to be
sparser than the exact LCIs computed by our separation routine.
This can have different implications on the solving process.
For example, within the subset~$(1000,7200)$, we observed an instance
(neos-1456979) for which the number of separated knapsack inequalities in
the settings \texttt{3-4-0SR} and \texttt{3-4-SSR} deviated only slightly.
In the former case, approximately~555 LPs needed to be solved per
node of the branch-and-bound tree, whereas in the second setting
approximately~1660 LPs needed to be solved.
Our denser LCIs therefore presumably create LPs that are more difficult to
solve.
For another instance (neos-555884), we noted that \scip spends more time
separating cutting planes at the root node within setting \texttt{3-4-0SR}
than in setting \texttt{3-4-SSR}.
This caused that the root node had a much better dual bound in the former
setting than in the latter setting.
Since \scip separates most cutting planes at the root node and not within
the branch-and-bound tree, setting \texttt{3-4-SSR} had troubles improving
the dual bound within the tree.
That is, although more and potentially stronger cutting planes are
separated when our separation routine is enabled, side effects within the
solver can cause that this results in a worse solving time.

\paragraph{Conclusions}

In this paper, we proposed to treat sparse knapsacks differently than
general knapsacks, because they admit a polynomial time separation algorithm
for LCIs.
Our goal was to investigate whether the special treatment allows to solve
general MIPs containing sparse knapsacks faster.
Based on our experiments, we could show that there is indeed a difference
between sparse knapsacks and general knapsacks.
The former greatly benefit from separating cutting planes within the
branch-and-bound tree, whereas the latter can be handled more effectively
by separating cutting planes only at the root node.
A potential explanation for this behavior is that we are currently missing
strong cutting planes for general knapsacks, i.e., the increase of the size of
LP relaxations caused by separated cutting planes is not compensated by the
tightened feasible region.
This explanation is supported by our experiments for the exact separation
of LCIs for knapsacks of sparsity~4, because in particular the hard
instances greatly benefit from our exact separation mechanism.
For~3-sparse knapsacks though, our exact separation algorithm seems to
hinder branch-and-bound solvers, possibly because LCIs are denser than
partially lifted inequalities.
To better understand the effect of exact separation for sparse knapsacks,
the following directions would be interesting for future research.
On the one hand, we noted that \scip's cutting planes for very sparse
knapsacks ($\sparsity = 3$) are already very effective, whereas we can
benefit from an exact separation of LCIs for knapsacks with~$\sparsity = 4$.
It would thus be interesting to investigate whether an exact separation
for knapsacks with an even higher~$\sparsity$-value further improves upon
the performance of the heuristically separated cutting planes.
On the other hand, we discussed, next to LCIs, also LCIs that incorporate GUB
information.
Since GUBs are not part of a sparse knapsacks itself, but rather arise from
additional problem structure, GUB-LCIs cannot be parameterized just based
on the coefficients of the knapsacks.
It would therefore be interesting to develop means to enhance
(parameterized) LCIs with GUB information in the most effective way.

Next to the separation algorithms of LCIs, we also discussed extended
formulations to model separation polyhedra.
Our numerical results indicated, however, that we can not expect an
improvement of running times when replacing separation algorithms by
extended formulations.
A possible explanation is that the extended formulations increase the
problem size too much without sufficiently strengthening the LP relaxation.
We note, however, that for some applications extended formulations of
particular symmetry handling constraints could be used successfully~\cite{ValidiBuchanan2022}.
Those extended formulations do not only handle symmetries, but also exploit further
problem information.
It would thus be interesting to investigate whether a coupling of extended
formulations of separation polyhedra with additional problem information
(such as GUBs) allows to strengthen the LP relaxation sufficiently such
that separation algorithms can be replaced by extended formulations.
This is out of scope of this article though.

\bigskip
\textbf{Acknowledgements}
We thank two anonymous reviewers for their valuable feedback.

\clearpage
\appendix
\section{Statistics on Sparse Knapsacks}
\label{sec:appendix}

In the introduction, we have provided statistics on the sparsity of
knapsacks arising in MIPLIB~2017.
To find these numbers, we have modified the code described
in Section~\ref{sec:numerics} such that it only collects the sparsity
information of a knapsack but does not exploit sparsity when separating
(lifted) cover inequalities.
Using the same infrastructure as for the experiments in
Section~\ref{sec:numerics}, the modified code has been run on all instances
of MIPLIB~2017 with a time limit of two hours, and we collected the
sparsity information of all knapsacks that have been identified by \scip
during the solving process.
The reported numbers thus do not only contain information about the
knapsack constraints present in the original MIPLIB instances, but also
about knapsacks that have been added while solving the problems, e.g., by
transforming cutting planes into constraints.

\bibliographystyle{elsarticle-num}

\begin{thebibliography}{10}
\expandafter\ifx\csname url\endcsname\relax
  \def\url#1{\texttt{#1}}\fi
\expandafter\ifx\csname urlprefix\endcsname\relax\def\urlprefix{URL }\fi
\expandafter\ifx\csname href\endcsname\relax
  \def\href#1#2{#2} \def\path#1{#1}\fi

\bibitem{doig1960algorithms}
A.~H. Land, A.~G. Doig, An automatic method of solving discrete programming
  problems, Econometrica 28~(3) (1960) 497--520.

\bibitem{bixby2012brief}
R.~E. Bixby, A brief history of linear and mixed-integer programming
  computation, Documenta Mathematica 2012 (2012) 107--121.

\bibitem{balas1975facets}
E.~Balas, Facets of the knapsack polytope, Mathematical programming 8 (1975)
  146--164.

\bibitem{boyd1993generating}
E.~A. Boyd, Generating {Fenchel} cutting planes for knapsack polyhedra, SIAM
  Journal on Optimization 3~(4) (1993) 734--750.

\bibitem{boyd1994fenchel}
E.~A. Boyd, {Fenchel} cutting planes for integer programs, Operations Research
  42~(1) (1994) 53--64.

\bibitem{crowder1983solving}
H.~Crowder, E.~L. Johnson, M.~W. Padberg, Solving large-scale zero-one linear
  programming problems, Operations Research 31~(5) (1983) 803--834.

\bibitem{gu1999lifted}
Z.~Gu, G.~L. Nemhauser, M.~W. Savelsbergh, Lifted cover inequalities for 0-1
  integer programs: Complexity, INFORMS Journal on Computing 11~(1) (1999)
  117--123.

\bibitem{hojny2020knapsack}
C.~Hojny, T.~Gally, O.~Habeck, H.~L{\"u}then, F.~Matter, M.~E. Pfetsch,
  A.~Schmitt, Knapsack polytopes: a survey, Annals of Operations Research 292
  (2020) 469--517.

\bibitem{wolsey2014integer}
L.~A. Wolsey, G.~L. Nemhauser, Integer and combinatorial optimization, John
  Wiley \& Sons, 2014.

\bibitem{balas1978facets}
E.~Balas, E.~Zemel, Facets of the knapsack polytope from minimal covers, SIAM
  Journal on Applied Mathematics 34~(1) (1978) 119--148.

\bibitem{wolsey1975faces}
L.~A. Wolsey, Faces for a linear inequality in 0--1 variables, Mathematical
  Programming 8~(1) (1975) 165--178.

\bibitem{padberg1975note}
M.~W. Padberg, A note on zero-one programming, Operations Research 23~(4)
  (1975) 833--837.

\bibitem{del2023complexity}
A.~Del~Pia, J.~Linderoth, H.~Zhu, On the complexity of separating cutting
  planes for the knapsack polytope, Mathematical Programming (2023) 1--27.

\bibitem{kaparis2010separation}
K.~Kaparis, A.~N. Letchford, Separation algorithms for 0-1 knapsack polytopes,
  Mathematical Programming 124 (2010) 69--91.

\bibitem{MIPLIB}
A.~Gleixner, G.~Hendel, G.~Gamrath, T.~Achterberg, M.~Bastubbe, T.~Berthold,
  P.~M. Christophel, K.~Jarck, T.~Koch, J.~Linderoth, M.~L\"ubbecke, H.~D.
  Mittelmann, D.~Ozyurt, T.~K. Ralphs, D.~Salvagnin, Y.~Shinano,
  \href{https://doi.org/10.1007/s12532-020-00194-3}{{MIPLIB} 2017: Data-driven
  compilation of the 6th mixed-integer programming library}, Mathematical
  Programming Computation (2021).
\newblock \href {https://doi.org/10.1007/s12532-020-00194-3}
  {\path{doi:10.1007/s12532-020-00194-3}}.
\newline\urlprefix\url{https://doi.org/10.1007/s12532-020-00194-3}

\bibitem{bolusani2024scip}
S.~Bolusani, M.~Besançon, K.~Bestuzheva, A.~Chmiela, J.~Dionísio,
  T.~Donkiewicz, J.~van Doornmalen, L.~Eifler, M.~Ghannam, A.~Gleixner,
  C.~Graczyk, K.~Halbig, I.~Hedtke, A.~Hoen, C.~Hojny, R.~van~der Hulst,
  D.~Kamp, T.~Koch, K.~Kofler, J.~Lentz, J.~Manns, G.~Mexi, E.~Mühmer, M.~E.
  Pfetsch, F.~Schlösser, F.~Serrano, Y.~Shinano, M.~Turner, S.~Vigerske,
  D.~Weninger, L.~Xu, The {SCIP} {Optimization} {Suite} 9.0 (2024).
\newblock \href {http://arxiv.org/abs/2402.17702} {\path{arXiv:2402.17702}}.

\bibitem{zemel1989easily}
E.~Zemel, Easily computable facets of the knapsack polytope, Mathematics of
  Operations Research 14~(4) (1989) 760--764.

\bibitem{chen2021complexity}
W.-K. Chen, Y.-H. Dai, On the complexity of sequentially lifting cover
  inequalities for the knapsack polytope, Science China Mathematics 64 (2021)
  211--220.

\bibitem{easton2008simultaneously}
T.~Easton, K.~Hooker, Simultaneously lifting sets of binary variables into
  cover inequalities for knapsack polytopes, Discrete Optimization 5~(2) (2008)
  254--261.

\bibitem{letchford2019lifted}
A.~N. Letchford, G.~Souli, On lifted cover inequalities: A new lifting
  procedure with unusual properties, Operations Research Letters 47~(2) (2019)
  83--87.

\bibitem{marchand2002cutting}
H.~Marchand, A.~Martin, R.~Weismantel, L.~Wolsey, Cutting planes in integer and
  mixed integer programming, Discrete Applied Mathematics 123~(1--3) (2002)
  397--446.

\bibitem{prasad2024newsequenceindependentliftingtechniques}
S.~Prasad, E.~Vitercik, M.-F. Balcan, T.~Sandholm,
  \href{https://arxiv.org/abs/2401.13773}{New sequence-independent lifting
  techniques for cutting planes and when they induce facets} (2024).
\newblock \href {http://arxiv.org/abs/2401.13773} {\path{arXiv:2401.13773}}.
\newline\urlprefix\url{https://arxiv.org/abs/2401.13773}

\bibitem{wolsey1977valid}
L.~A. Wolsey, Valid inequalities and superadditivity for 0--1 integer programs,
  Mathematics of Operations Research 2~(1) (1977) 66--77.

\bibitem{peled1977properties}
U.~N. Peled, Properties of facets of binary polytopes, in: Annals of Discrete
  Mathematics, Vol.~1, Elsevier, 1977, pp. 435--456.

\bibitem{hartvigsen1992complexity}
D.~Hartvigsen, E.~Zemel, The complexity of lifted inequalities for the knapsack
  problem, Discrete Applied Mathematics 39~(2) (1992) 113--123.

\bibitem{fereirra1994combinatorial}
C.~E. Fereirra, On combinatorial optimization problems arising in computer
  system design, Ph.D. thesis, Zuse Institute Berlin (ZIB) (1994).

\bibitem{klabjan1998complexity}
D.~Klabjan, G.~L. Nemhauser, C.~Tovey, The complexity of cover inequality
  separation, Operations Research Letters 23~(1-2) (1998) 35--40.

\bibitem{hoffman1991improving}
K.~L. Hoffman, M.~W. Padberg, Improving {LP}-representations of zero-one linear
  programs for branch-and-cut, ORSA Journal on Computing 3~(2) (1991) 121--134.

\bibitem{hickman2015merging}
R.~Hickman, T.~Easton, Merging valid inequalities over the multiple knapsack
  polyhedron, International Journal of Operational Research 24~(2) (2015)
  214--227.

\bibitem{padberg19801}
M.~W. Padberg, (1, k)-configurations and facets for packing problems,
  Mathematical Programming 18 (1980) 94--99.

\bibitem{gottlieb1988facets}
E.~S. Gottlieb, M.~Rao, Facets of the knapsack polytope derived from disjoint
  and overlapping index configurations, Operations Research Letters 7~(2)
  (1988) 95--100.

\bibitem{dietrich1992tightening}
B.~L. Dietrich, L.~F. Escudero, On tightening cover induced inequalities,
  European Journal of Operational Research 60~(3) (1992) 335--343.

\bibitem{atamturk2005cover}
A.~Atamt{\"u}rk, Cover and pack inequalities for (mixed) integer programming,
  Annals of Operations Research 139~(1) (2005) 21--38.

\bibitem{weismantel19970}
R.~Weismantel, On the 0/1 knapsack polytope, Mathematical Programming 77 (1997)
  49--68.

\bibitem{glover1997generating}
F.~Glover, H.~D. Sherali, Y.~Lee, Generating cuts from surrogate constraint
  analysis for zero-one and multiple choice programming, Computational
  Optimization and Applications 8 (1997) 151--172.

\bibitem{andrew1992pseudopolynomial}
E.~A. Boyd, A pseudopolynomial network flow formulation for exact knapsack
  separation, Networks 22~(5) (1992) 503--514.

\bibitem{padberg1973facial}
M.~W. Padberg, On the facial structure of set packing polyhedra, Mathematical
  Programming 5~(1) (1973) 199--215.

\bibitem{nemhauser1994lifted}
G.~L. Nemhauser, P.~H. Vance, Lifted cover facets of the 0--1 knapsack polytope
  with {GUB} constraints, Operations Research Letters 16~(5) (1994) 255--263.

\bibitem{riise2016recursive}
A.~Riise, C.~Mannino, L.~Lamorgese, Recursive logic-based {Benders’}
  decomposition for multi-mode outpatient scheduling, European Journal of
  Operational Research 255~(3) (2016) 719--728.

\bibitem{conforti2010extended}
M.~Conforti, G.~Cornu{\'e}jols, G.~Zambelli, Extended formulations in
  combinatorial optimization, 4OR 8~(1) (2010) 1--48.

\bibitem{cormen2001introduction}
T.~H. Cormen, C.~E. Leiserson, R.~L. Rivest, C.~Stein, Introduction to
  algorithms, Second Edition, MIT press: Cambridge, US, 2001.

\bibitem{goemans2015permutahedron}
M.~X. Goemans, Smallest compact formulation for the permutahedron, Mathematical
  Programming 153 (2015).

\bibitem{dantzig1967generalized}
G.~B. Dantzig, R.~M. Van~Slyke, Generalized upper bounding techniques, Journal
  of Computer and System Sciences 1~(3) (1967) 213--226.

\bibitem{AndersEtAl2023}
M.~Anders, P.~Schweitzer, J.~Stie{\ss},
  \href{https://doi.org/10.48550/arXiv.2302.06351}{Engineering a preprocessor
  for symmetry detection}, CoRR abs/2302.06351 (2023).
\newblock \href {http://arxiv.org/abs/2302.06351} {\path{arXiv:2302.06351}},
  \href {https://doi.org/10.48550/arXiv.2302.06351}
  {\path{doi:10.48550/arXiv.2302.06351}}.
\newline\urlprefix\url{https://doi.org/10.48550/arXiv.2302.06351}

\bibitem{JunttilaKaski2011}
T.~Junttila, P.~Kaski, Conflict propagation and component recursion for
  canonical labeling, in: A.~Marchetti{-}Spaccamela, M.~Segal (Eds.), Theory
  and Practice of Algorithms in (Computer) Systems -- First International
  {ICST} Conference, {TAPAS} 2011, Rome, Italy, April 18--20, 2011.
  Proceedings, Vol. 6595 of Lecture Notes in Computer Science, Springer, 2011,
  pp. 151--162.
\newblock \href {https://doi.org/10.1007/978-3-642-19754-3\_16}
  {\path{doi:10.1007/978-3-642-19754-3\_16}}.

\bibitem{release}
C.~Hojny, C.~Roy, Supplementary material for the article ``{Computational}
  aspects of lifted cover inequalities for knapsacks with few different
  weights'', \url{https://doi.org/10.5281/zenodo.14516189 } (2024).

\bibitem{KaibelLoos2011}
V.~Kaibel, A.~Loos, Finding descriptions of polytopes via extended formulations
  and liftings, in: A.~R. Mahjoub (Ed.), Progress in Combinatorial
  Optimization, Wiley, 2011.

\bibitem{HojnyPfetsch2019}
C.~Hojny, M.~E. Pfetsch,
  \href{https://doi.org/10.1007/s10107-018-1239-7}{Polytopes associated with
  symmetry handling}, Mathematical Programming 175 (2019) 197--240.
\newblock \href {https://doi.org/10.1007/s10107-018-1239-7}
  {\path{doi:10.1007/s10107-018-1239-7}}.
\newline\urlprefix\url{https://doi.org/10.1007/s10107-018-1239-7}

\bibitem{PfetschRehn2019}
M.~E. Pfetsch, T.~Rehn, A computational comparison of symmetry handling methods
  for mixed integer programs, Mathematical Programming Computation 11~(1)
  (2019) 37--93.
\newblock \href {https://doi.org/10.1007/s12532-018-0140-y}
  {\path{doi:10.1007/s12532-018-0140-y}}.

\bibitem{Hojny2020}
C.~Hojny,
  \href{http://www.sciencedirect.com/science/article/pii/S0167637720301103}{Polynomial
  size {IP} formulations of knapsack may require exponentially large
  coefficients}, Operations Research Letters 48~(5) (2020) 612--618.
\newblock \href {https://doi.org/https://doi.org/10.1016/j.orl.2020.07.013}
  {\path{doi:https://doi.org/10.1016/j.orl.2020.07.013}}.
\newline\urlprefix\url{http://www.sciencedirect.com/science/article/pii/S0167637720301103}

\bibitem{ValidiBuchanan2022}
H.~Validi, A.~Buchanan, Political districting to minimize cut edges,
  Mathematical Programming Computation 14 (2022) 623--672.
\newblock \href {https://doi.org/10.1007/s12532-022-00221-5}
  {\path{doi:10.1007/s12532-022-00221-5}}.

\end{thebibliography}



\end{document}